\documentclass[reqno,english]{amsart}
\usepackage{amsfonts,amsmath,latexsym,verbatim,amscd,mathrsfs,color,array}

\usepackage{amsmath,amssymb,amsthm,amsfonts,graphicx,color}
\usepackage{amssymb}
\usepackage{pdfsync}
\usepackage{epstopdf}
\usepackage[colorlinks=true]{hyperref}

\makeatletter
\def\@currentlabel{2.1}\label{e:dispaa}
\def\@currentlabel{2.21}\label{e:dispau}
\def\@currentlabel{2.22}\label{e:dispav}
\def\@currentlabel{2.23}\label{e:dispaw}
\def\@currentlabel{2.24}\label{e:dispax}
\makeatother

\usepackage{amsmath,  amsthm, amssymb}

\DeclareMathOperator{\sech}{sech}

\newtheorem{thm}{Theorem}[section]
\newtheorem{lem}[thm]{Lemma}

\newtheorem{prop}[thm]{Proposition}

\newtheorem{rem}[thm]{Remark}

\numberwithin{equation}{section}

\newcommand{\R}{\mathbb{R}}

\newcommand{\ve}{\varepsilon}

\begin{document}

\title[]{The existence and stability of spike solutions for a chemotaxis system modeling  crime pattern formation }

\author{Henri Berestycki}

\address{\'{E}cole des Hautes \'{E}tudes en Sciences Sociales, PSL University Paris, Centre d'analyse et de math\'{e}matique sociales (CAMS), CNRS, 54 bouvelard Raspail, 75006, Paris, France}

\email{hb@ehess.fr}
\author[L. Mei]{Linfeng Mei}
\address{Department of Mathematics and Computer Sciences, Zhejiang Normal University,  Jinhuan 321004,  P.R. China}
\email{mlf@htu.edu.cn}

\author[J. Wei]{Juncheng Wei}
\address{Department of Mathematics,
The University of British Columbia,
Vancouver, BC
Canada V6T 1Z2}
\email{jcwei@math.ubc.ca}

\thanks{The research of H. Berestycki has been supported by the French National Research Agency (ANR), within  project NONLOCAL ANR-14-CE25-0013. L. Mei is supported by the National Natural Science Foundation of China: grant no.11771125. J. Wei is partially supported by NSERC of Canada.}

\begin{abstract}
This paper is a continuation of the paper Berestycki, Wei and Winter \cite{Berestycki2014}. In \cite{Berestycki2014},  the existence of multiple symmetric and asymmetric spike solutions of a chemotaxis system modeling crime pattern formation, suggested by Short, Bertozzi, and Brantingham \cite{Short2010}, has been proved in the one-dimensional case. The problem of stability of these spike solutions has been left open.  In this paper, we establish the existence of a single radial symmetric spike solution for the system in the one and two-dimensional cases. The main difficulty is to deal with quasilinear elliptic problems whose diffusion coefficients vary largely near the core. We also study the linear stability of the spike solutions in both one-dimensional and two-dimensional cases which show complete different behaviors. In the one-dimensional case, we show that when the reaction time ratio $\tau>0$ is small enough, or large enough, the spike solution is linearly stable. In the two-dimensional case, when $\tau$ is small enough, the spike solution is linearly stable; while when $\tau$ is large enough, the spike solution is linearly unstable and Hopf bifurcation occurs from the spike solution at some $\tau=\tau_h$.
\end{abstract}

\date{\today}

\maketitle

\section{Introduction}

Pattern forming reaction-diffusion systems have been and are applied to many phenomena in the natural sciences. Recent works have also started to use such systems to describe macroscopic social phenomena. In this direction, Short, Bertozzi and Brantingham \cite{Short2010} have proposed a system of nonlinear parabolic partial differential equations to describe  the formation of hotspots of criminal activity. Their equations are derived from an agent-based lattice model which incorporates the movement of criminals and a given scalar field representing the attractiveness of crime in a given location.
Let $\Omega$ be a bounded smooth domain in $\R^N (N=1, 2)$. Then the system in $\Omega$ reads
\begin{equation}\label{eqn_t_0}
\left\{\begin{aligned}
&A_t=\ve^2 \Delta A-A+PA+\alpha_0(x), \;\;&&x\in\Omega,\;t>0,\\
&\tau P_t=D(\ve)\nabla (\nabla P-2\frac{P}{A}\nabla A)-PA+\gamma_0(x), \;\;&& x\in\Omega,\; t>0,\\
&\partial_{n} A=\partial_{n} P=0,\;\;&&x\in\partial\Omega, \;t>0.
\end{aligned}\right.
\end{equation}
Here $A=A(x, t)$ is the criminal activity  at the place $x$ and the time $t$, and $P=P(x, t)$ denotes the density of criminals at $(x, t)$. The field $A(x,t)$ represents a variable incorporating the perceived criminal opportunities.  The rate at which crimes occur is given by $PA$. When this rate increases, the number of criminals is reduced while the attractiveness increases. The latter feature corresponds to repeated offences. The positive function $\alpha_0(x)$ is the intrinsic attractiveness which is static in time but possibly variable in space. The positive function $\gamma_0(x)$ is the introduction rate of the offenders. For the precise meanings of the functions $\alpha_0(x)$ and $\gamma_0(x)$, we refer to \cite{Short2008, Short2010, Short2010a} and the references therein. The small parameter $\ve>0$ is assumed to be independent of $x$ and $t$. The parameter $\ve^2$ represents nearest neighbor interactions in the lattice model for the attractiveness. We assume that it is very small which corresponds to the temporal dependence of attractiveness dominating its spatial dependence. This is related to  the slow propagation of the attractiveness, as compared to the propagation rate of the criminals
 $D(\ve)>0$. Here $D(\ve)$ is a large positive constant that does not depend on $x$ and $t$ and tends to $+\infty$ at a suitable speed, as $\ve\to 0$.

\medskip

 The parameter $
\tau>0$ describes the ratio of the reaction times of the two equations. We assume that $\tau>0$ depends on $\ve$: $\tau=\tau(\ve)$. As it turns out the natural scaling for  $\tau(\ve)$ is $\tau(\ve)\sim \ve^{-N}$.  (In other cases one gets $0$ or $+\infty$ as limits.) Therefore in this paper we assume
\begin{equation}\label{tau}
\tau(\ve)\ve^{N} =O(1).
\end{equation}

Clearly if
\begin{equation}\label{alpha-gamma}
\alpha_0(x)\equiv \alpha_0,\qquad \gamma_0(x)\equiv \gamma_0,
\end{equation}
then
\begin{equation}\label{constSol}
(A, P)=\left(\alpha_{0}+\gamma_0, \frac{\gamma_0}{\alpha_{0}+\gamma_0}\right)
\end{equation}
is the only constant steady state solution, which does not depend on $\ve$.

We are interested in the steady state solutions of \eqref{eqn_t_0} with hotspot (spike) pattern, and its linear stability, when $D(\ve)\to \infty$ at a suitable speed as $\ve\to 0$.

\medskip

Before going into this, let us
mention some  related mathematical works. Short, Bertozzi and Brantingham \cite{Short2010} proposed the model \eqref{eqn_t_0} on mean field considerations. They also performed a weakly nonlinear analysis around the constant solution, assuming that \eqref{alpha-gamma} holds. Cantrell, Cosner and Manasevich \cite{Cantrell2012} considered rigorously the global bifurcation of steady states emanating from the unique constant steady state \eqref{constSol}.  Rodriguez and Winkler \cite{RW} and  Winkler \cite{Winkler} established the existence of globally defined solutions to the system \eqref{eqn_t_0} in a one dimensional interval or  two-dimensional ball respectively, assuming radially symmetric initial conditions. Kolokolnikov, Ward and Wei \cite{Kolokolnikov2014}, and Ward and Tse \cite{Tse2016} studied the existence and stability of multiple symmetric spikes for the steady states of \eqref{eqn_t_0} via matched asymptotics.  See also  Lloyd and O'Farrell \cite{Lloyd2013, Lloyd2016} by  geometric singular perturbations. Furthermore, Berestycki, Wei and Winter \cite{Berestycki2014} gave a  rigorous proof of the existence of symmetric and asymmetric multi-spike steady state solutions in a one dimensional interval by reducing the problem to a  Schnakenberg type system. But they left open the stability of the spike solutions.

\medskip

We would like to mention that  Zipkin, Short and Bertozzi \cite{Zipkin2014} and  Ward and Tse \cite{Tse2018} studied crime models along the same line but with police intervention. Chaturapruek et al. \cite{Chaturapruek2013} analyzed  a crime model with Levy flights. Berestycki and Nadal \cite{Berestycki2010} proposed and analyzed another model of criminality  with hotspot phenomena. Finally, Berestycki, Rodriguez, and Ryzhik \cite{Berestycki2013} proved the existence of traveling wave solutions in a crime model.

\medskip

Before stating our main results,
let us  make the  change of variable
\begin{equation}
V=P/A^2,
\end{equation}
and transform \eqref{eqn_t_0} into an equivalent form:
\begin{equation}\label{eqn_t_1}
\left\{\begin{aligned}
&A_t=\ve^2 \Delta A-A+VA^3+\alpha_{0}(x), \;\;&&x\in\Omega,\; t>0,\\
&\tau (A^2V)_t=D(\ve)\nabla (A^2\nabla V)-VA^3+\gamma_0(x), \;\; &&x\in\Omega,\; t>0,\\
&\partial_{n} A=\partial_{n} V=0,\;\;&&x\in\partial\Omega, \; t>0.
\end{aligned}\right.
\end{equation}
By the  rescaling
\begin{equation}
 A(x, t)=\ve^{-N}u(x, t), \hskip 1cm V(x, t)=\ve^N v(x, t),
\end{equation}
\eqref{eqn_t_1} becomes
\begin{equation}\label{eqn0}
\left\{\begin{aligned}
&u_t=\ve^2\Delta u-u+vu^3+\alpha_{0}(x)\ve^{N}, \;\;&&x\in\Omega,\\
&\tau (u^2v)_t=D(\ve)\nabla (u^2\nabla v)-\ve^{-N}vu^3+\gamma_0(x), \;\; &&x\in\Omega,\\
&\partial_{n} u=\partial_{n} v=0,\quad &&x\in \partial\Omega.
\end{aligned}\right.
\end{equation}

We would like to construct  spiky positive steady states of \eqref{eqn0} concentrating at some chosen finite spots in $\Omega$.  The core profile of the spiky solution is governed by the radially symmetric solution to  the problem
\begin{equation}\label{eq-w}
\Delta w-w+w^3=0\qquad\text{in}\;\;\R^N.
\end{equation}
It is well-known  \cite{Gidas1981, Kwong1991} that for $N\leq 3$ \eqref{eq-w} has a unique solution satisfying
\begin{equation}\label{w}
w>0\quad\text{in}\;\;\R^N,\qquad w(0)=\max_{y\in\R^N}w(y),\qquad \lim_{|y|\to +\infty}w(y)=0.
\end{equation}

For  solutions concentrating at the single spot $x_0\in \Omega$, we  expect them to have the profile
\begin{equation}\label{profile1}
u_{\ve}(x)\sim \alpha_{0}(x)\ve^N+[v_{\ve}(x_\ve)]^{-1/2}w\left(\frac{x-x_\ve}{\ve}\right),\qquad v_{\ve}(x)\sim v_{\ve}(x_\ve),
\end{equation}
where $x_{\ve}\to x_0$ as $\ve\to 0$, and $w$ is the unique positive solution of \eqref{eq-w} satisfying \eqref{w}.

\medskip

Assuming \eqref{profile1} and integrating the  steady state equations of \eqref{eqn0} over $\Omega$ we obtain
\begin{equation}\label{profile-v}
v_{\ve}(x)\sim \left(\frac{\int_{\R^N}w^3(y)dy}{\int_{\Omega}\gamma_0(x)dx}\right)^2.
\end{equation}
Actually, we will construct a solution with the profile
\begin{equation}\label{u1-}
u_{\ve}(x)=\alpha_{0}\ve^N+[v_{\ve}(x_\ve)]^{-1/2}w\left(\frac{x-x_\ve}{\ve}\right)+\phi(x)
\end{equation}
with $\phi(x)$ satisfying
\begin{equation}\label{u2-}
|\phi(x)|\leq C\ve^{1+N}\max(e^{-\frac{|x-x_\ve|}{2\ve}}, \sqrt{\ve}),
\end{equation}
where $C>0$ is some constant, independent of $v$ and $\epsilon>0$,  to be properly chosen.  In this case $v_{\ve}$ has  the profile \eqref{profile-v}.

\medskip

If $\Omega$ is a generic bounded domain in $\R^N$, the construction of spiky positive solutions seems rather difficult, due to the quasilinear nature of the problem. As a model problem we consider the case when $\Omega$ is a ball in $\R^N$ and construct  radial spike solutions concentrating at the center of the ball. In doing so, we always assume that $\alpha_0(x)$ and $\gamma_0(x)$ are positive constant functions, namely \eqref{alpha-gamma} that  holds.

\medskip

The main findings  of this paper can be summarized in the following three theorems.
\begin{thm}\label{thm1}
Let $N=1, 2$. Assume  $\Omega=B_R\subset \R^N$ and \eqref{alpha-gamma} holds. Assume that $ D(\epsilon)$ satisfies
\begin{equation}\label{D-ve}
D(\ve)=\frac{D_0(\ve)}{\ve^{2N}}\qquad \text{with}\;\;\;\;D_0(\ve)\to \infty\quad\text{as}\;\;\ve\to 0.
\end{equation}
Then, as $\ve\to 0$, problem \eqref{eqn_t_1} has a radial symmetric steady state $(A_{\ve}, V_{\ve})$ satisfying the following properties
\begin{equation}
A_{\ve}(x)=\alpha_0+\frac{1}{\ve^N}\frac{1}{\sqrt{v_0}}w\left(\frac{x}{\ve}\right)+O(\ve),
\end{equation}
\begin{equation}
V_{\ve}(x)=v_0\ve^N+O\left(\frac{\ve^N}{D_0(\ve)}\right),
\end{equation}
where
\begin{equation}
v_0:=\left(\frac{\int_{\R^N}w^3(y)dy}{\gamma_0|B_R|}\right)^2.
\end{equation}
\end{thm}

\begin{thm}\label{stability-2d}
Let $N=2$. Assume the conditions in Theorem \ref{thm1} hold. Then, for all small $\ve>0$, there exists $0<\tilde{\tau}_{1, \ve}<\tilde{\tau}_{2, \ve}<\infty$ such that for all $0<\tau\leq\tilde{\tau}_{1, \ve}\ve^{-N}$, the spike solution of Theorem \ref{thm1} is stable, while for $\tau\geq \tilde{\tau}_{2, \ve}\ve^{-N}$, the spike solution of Theorem \ref{thm1} is unstable, and Hopf bifurcation occurs at some $\tau_{h,\ve}\in (\tilde{\tau}_{1, \ve}\ve^{-N}, \tilde{\tau}_{2, \ve}\ve^{-N})$.
\end{thm}

\begin{thm}\label{stability-1d}
Let $N=1$. Assume the conditions in Theorem \ref{thm1} hold. Then, for all small $\ve>0$, there exists $0<\tilde{\tau}_{1, \ve}\leq \tilde{\tau}_{2, \ve}<\infty$ such that for all $0<\tau\leq\tilde{\tau}_{1, \ve}\ve^{-N}$ and $\tau\geq \tilde{\tau}_{2,\ve}\ve^{-N}$ the spike solution of Theorem \ref{thm1} is stable.
\end{thm}

\medskip

Theorem \ref{stability-1d}  suggests that in the one dimensional case, the spike solution of Theorem \ref{thm1} is stable for all $\tau>0$. This is in sharp contrast to the two dimensional case, as depicted by theorem \ref{stability-2d}, when Hopf bifurcation occurs at some $\tau=\tau_{h,\ve}\in (0, \infty)$.

\medskip

The rest of the paper is organized as follows. In Section 2, we first collect some preliminary facts, which  play important roles in the rest of the paper. Then we reduce the system for the steady state $(u_{\ve}, v_{\ve})$ to a single equation by showing that $v_{\ve}$ is almost flat. Section 3 is used to derive a nonlocal eigenvalue problem, which provides the basis for the stability (and nondegeneracy) analysis for the spike solution we are going to construct. In Section 4 we prove Theorem \ref{thm1}. Section 5 and 6 are devoted to the proof of Theorem \ref{stability-2d} and Theorem \ref{stability-1d}, respectively.

Throughout this paper we always assume that $N=1,2$.

\section{Reduction to a single equation}\label{sec:redToSin}
Let $N=1, 2$ and $w$ be the unique solution satisfying (\ref{eq-w})-(\ref{w}).
We also recall that $w'(|y|)<0$ for $|y|>0$, and there is a constant $A_N>0$ such that
\begin{equation}
w(r)=A_N r^{-\frac{N-1}{2}}e^{-r}\left(1+O\left(r^{-1}\right)\right)\qquad\text{as}\;\;r=|y|\to +\infty,
\end{equation}
\begin{equation}
w'(r)=-A_N r^{-\frac{N-1}{2}}e^{-r}\left(1+O\left(r^{-1}\right)\right)
\qquad\text{as}\;\;r=|y|\to +\infty.
\end{equation}

From the energy identity
\[
\int_{\R^N}|\nabla w|^2+\int_{\R^N}w^2-\int_{\R^N}w^4=0,
\]
and the Pohozaev identity
\begin{equation}
\frac{2-N}{2N}\int_{\R^N}|\nabla u|^2-\frac{1}{2}\int_{\R^N}w^2+\frac{1}{4}\int_{\R^N}w^4=0,
\end{equation}
we  have
\begin{equation}
\int_{\R^N}w^4=\frac{4}{4-N}\int_{\R^N}w^2,\qquad\int_{\R^N}|\nabla w|^2=\frac{N}{4-N}\int_{\R^N}w^2.
\end{equation}
Direct integration of  equation \eqref{eq-w}  yields
\begin{equation}
\int_{\R^N}w=\int_{\R^N}w^3.
\end{equation}

Let us denote
\begin{equation}\label{L0}
L_0[\phi]=\Delta \phi-\phi+3w^2\phi,\qquad \phi\in H^1(\R^N).
\end{equation}
Then we have the following well-known results (Theorem 2.1 of \cite{Lin1988} and Lemma C of \cite{Ni1992}).
\begin{lem}\label{lem:L0}
The eigenvalue problem
\begin{equation}
L_0[\phi]=\mu\phi,\qquad\phi\in H^1(\R^N),
\end{equation}
admits the following sets of eigenvalues
\begin{equation}
\mu_0>0, \quad \mu_1=\mu_2=\cdots=\mu_{N}=0, \quad\mu_{N+1}<0,\cdots.
\end{equation}
The eigenfunction $\phi_0$ corresponding to $\mu_0$ is simple and can be made positive and radial symmetric; the space of eigenfunctions corresponding to the eigenvalue $0$ is
\begin{equation}
X_0=Kernel(L_0):=\text{Span}\left\{\frac{\partial w}{\partial y_j}\Big| j=1, \cdots, N\right\}.
\end{equation}
\end{lem}

Denote \[w_0:=\frac{1}{2}w+\frac{1}{2}y\nabla w.\]
Direct calculation yields
\begin{equation}
L_0[w]=2w^3,\hskip 2cm L_0 [w_0]=w,
\end{equation}
and
\begin{equation}
\int_{\R^N}wL^{-1}_0[w]=\int_{\R^N}ww_0=\left(\frac{1}{2}-\frac{N}{4}\right)\int_{\R^N}w^2,
\end{equation}
\begin{equation}
\int_{\R^N}w^3L^{-1}_0[w]=\int_{\R^N}w^3w_0=\frac{1}{2}\int_{\R^N}w^2.
\end{equation}

In this paper we use $B_r$ to denote the open ball in $\R^N$ centred at the origin and with radius $r\in (0, \infty)$. Let us assume $\Omega=B_R$ for some fixed $R\in (0, \infty)$. We also assume that the diffusion coefficient of the equation for $v_{\ve}$ is suitably large, that is \eqref{D-ve} holds.
Then the steady state of \eqref{eqn0} satisfies the system
\begin{equation}\label{eq}
\left\{\begin{aligned}
&\ve^2\Delta u-u+vu^3+\alpha_{0}\ve^N=0, &&x\in B_R,\\
&\frac{D_0(\ve)}{\ve^{2N} }\nabla (u^2 \nabla v)-\ve^{-N}vu^3+\gamma_0=0, \;\;&&x\in B_R,\\ &\partial_{n}u=\partial_{n}v=0,\quad &&x\in \partial B_R.
\end{aligned}\right.
\end{equation}
We only consider radial solutions. Under this assumption, \eqref{eq} is equivalent to
\begin{equation}\label{eq-r}
\left\{\begin{aligned}
&\ve^2\nabla_r(r^{N-1}\nabla_ru)=r^{N-1}(u-vu^3-\alpha_{0}\ve^N), \quad\;\;&&0<r<R,\\
&\frac{D_0(\ve)}{\ve^{2N}}\nabla_r(r^{N-1}u^2 \nabla_rv)=r^{N-1}(\ve^{-N}vu^3-\gamma_0),\quad \;\; &&0<r<R,\\
&\nabla_ru(0)=\nabla_ru(R)=\nabla_rv(0)=\nabla_rv(R)=0,\\
\end{aligned}\right.
\end{equation}
where  $r=|x|$ and $\nabla_r$ denotes differentiation with respect to $r$. Here and in the rest of the paper, for a radial function $f(x)$, we abuse the notation a bit and use $f(|x|)$ to denote the same function.

Given $u_{\ve}>0$, let $v_{\ve}$ be the unique solution of the following linear problem
\begin{equation}\label{eq-v}
\left\{\begin{aligned}
&\frac{D_0(\ve)}{\ve^{2N} }\nabla_r(r^{N-1}u_{\ve}^2 \nabla_rv_{\ve})=r^{N-1}(\ve^{-N}v_{\ve}u_{\ve}^3-\gamma_0), \;\; 0<r<R,\\
& \nabla_rv_{\ve}(0)=\nabla_rv_{\ve}(R)=0.
\end{aligned}\right.
\end{equation}
By the maximum principle, $v_{\ve}>0$.

Integrating the equation \eqref{eq-v} over $[0, r]$ for $r\in (0, R]$ we obtain
\begin{equation}\label{dv}
\nabla_rv_{\ve}(r)=\frac{\ve^{2N}}{D_0(\ve)}\frac{1}{r^{N-1}u_{\ve}^2(r)}\int_0^r[\ve^{-N}s^{N-1}v_{\ve}(s)u_{\ve}^3(s)-\gamma_0 s^{N-1}]ds.
\end{equation}
Let us assume $\|v_{\ve}\|_\infty$ is bounded away  from $0$ and $\infty$ as $\ve\to 0$, and
\begin{equation}\label{u-p1}
u_{\ve}(r)=\alpha_{0}\ve^N+[v_{\ve}(0)]^{-1/2}w\left(\frac{r}{\ve}\right)+\phi_{\ve}(r)
\end{equation}
with $\phi_{\ve}(r)$ satisfying
\begin{equation}\label{u-p2}
|\phi_{\ve}(r)|\leq C\ve^{N+1}\max(e^{-\frac{r}{2\ve}}, \sqrt{\ve}),
\end{equation}
where $C>0$ is some large constant.

A key observation about $v_{\ve}$ is the following estimate.
\begin{lem}\label{lem-est-v}
For any $v_{\ve}\in [c_1, c_2]$ with constants $c_1>0$ and $c_2>0$, if $u_{\ve}$ has the form \eqref{u-p1} and \eqref{u-p2}, then
\begin{equation}\label{dv3}
|\nabla_rv_{\ve}(r)|=O\left(\frac{1}{D_0(\ve)}\right) r,
\end{equation}
for all $r\in [0, R]$, and
as a consequence,
\begin{equation}\label{var-v}
|v_{\ve}(r)-v_{\ve}(0)|=O\left(\frac{1}{D_0(\ve)}\right)r^2,
\end{equation}
for all $r\in [0, R]$.
\end{lem}
\begin{proof}
We show \eqref{dv3} first. By the formula \eqref{dv} we have
\begin{equation}\label{dv2}
\nabla_rv_{\ve}(r)=-\frac{\gamma_0\ve^{2N}r}{ND_0(\ve)u^2_{\ve}(r)}+\frac{\ve^N}{D_0(\ve)r^{N-1}u^2_{\ve}(r)}\int_0^rs^{N-1}v_{\ve}(s)u^3_{\ve}(s)ds.
\end{equation}

Since $u_{\ve}(r)\geq \frac{1}{2}\alpha_{0}\ve^N$, the first term on the right-hand side of \eqref{dv2} is easy to estimate:
\[\frac{\gamma_0\ve^{2N}r}{ND_0(\ve)u^2_{\ve}(r)}\leq \frac{4\gamma_0 r}{N\alpha_{0}^2D_0(\ve)}=O\left(\frac{1}{D_0(\ve)}\right)r.\]

To estimate the second term on the right-hand side of \eqref{dv2}, we divide the interval into several subintervals.

\begin{itemize}

\item For $r\in [0, \ve|\log{\ve}|/4)$, we have
\begin{equation}
u_{\ve}(r)\geq c\ve^{1/4}\qquad\text{and}\qquad  |u_{\ve}(r)|=O(1),
\end{equation}
and hence
\[
\frac{\ve^N}{D_0(\ve)r^{N-1}u^2_{\ve}(r)}\int_0^rs^{N-1}v_{\ve}(s)u^3_{\ve}(s)ds
\leq \frac{\|v_{\ve}\|_{\infty}\|u_{\ve}\|^3_{\infty}\ve^N}{D_0(\ve)r^{N-1}u^2_{\ve}(r)}\int_0^rs^{N-1}ds\]
\[=\frac{\|v_{\ve}\|_{\infty}\|u_{\ve}\|^3_{\infty}\ve^N}{ND_0(\ve)u^2_{\ve}(r)}r
\leq \frac{\|v_{\ve}\|_{\infty}\|u_{\ve}\|^3_{\infty}\ve^{N-\frac{1}{2}}}{ND_0(\ve)c^2}r\]
\[=O\left(\frac{\ve^{N-\frac{1}{2}}}{D_0(\ve)}\right)r.\]

\item For $r\in [\ve|\log{\ve}|/4, \ve|\log{\ve}|/2)$, we have
\begin{equation}
c_1\ve^{1/2}\leq u_{\ve}(r)\leq c_2\ve^{1/4}
\end{equation}
for some constants $c_1>0$ and $c_2>0$,
and hence
\[
\frac{\ve^N}{D_0(\ve)r^{N-1}u^2_{\ve}(r)}\int_0^rs^{N-1}v_{\ve}(s)u^3_{\ve}(s)ds\]
\[\leq \frac{\|v_{\ve}\|_{\infty}\ve^N}{D_0(\ve)r^{N-1}c_1^2\ve}\int_0^rc_2^3\ve^{3/4}s^{N-1}ds\]
\[=\frac{c_2^3\|v_{\ve}\|_{\infty}\ve^{N-\frac{1}{4}}}{N c_1^2D_0(\ve)}r
=O\left(\frac{\ve^{N-\frac{1}{4}}}{D_0(\ve)}\right)r.\]

\item For  $r\in [\ve|\log{\ve}|/2, 3\ve|\log{\ve}|/4)$, we have
\begin{equation}
c_1\ve^{3/4}\leq u_{\ve}(r)\leq c_2\ve^{1/2},
\end{equation}
for some constants $c_1>0$ and $c_2>0$,
and hence
\begin{equation}
\begin{aligned}
&\frac{\ve^N}{D_0(\ve)r^{N-1}u^2_{\ve}(r)}\int_0^rs^{N-1}v_{\ve}(s)u^3_{\ve}(s)ds\\
&\leq \frac{\|v_{\ve}\|_{\infty}\ve^N}{D_0(\ve)r^{N-1}c_1^2\ve^{3/2}}\int_0^rc_2^3\ve^{3/2}s^{N-1}ds\\
&=\frac{c_2^3\|v_{\ve}\|_{\infty}\ve^{N}}{N c_1^2D_0(\ve)}r
=O\left(\frac{\ve^{N}}{D_0(\ve)}\right)r.
\end{aligned}
\end{equation}

\item For  $r\in [3\ve|\log{\ve}|/4, \ve|\log{\ve}|)$, we have
\begin{equation}
c_1\ve\leq u_{\ve}(r)\leq c_2\ve^{3/4},
\end{equation}
for some constants $c_1>0$ and $c_2>0$,
and hence
\begin{equation}
\begin{aligned}
&\frac{\ve^N}{D_0(\ve)r^{N-1}u^2_{\ve}(r)}\int_0^rs^{N-1}v_{\ve}(s)u^3_{\ve}(s)ds\\
&\leq \frac{\|v_{\ve}\|_{\infty}\ve^N}{D_0(\ve)r^{N-1}c_1^2\ve^{2}}\int_0^rc_2^3\ve^{9/4}s^{N-1}ds\\
&=\frac{c_2^3\|v_{\ve}\|_{\infty}\ve^{N+\frac{1}{4}}}{N c_1^2D_0(\ve)}r
=O\left(\frac{\ve^{N+\frac{1}{4}}}{D_0(\ve)}\right)r.
\end{aligned}
\end{equation}

\item For  $r\in [\ve|\log{\ve}|, 5\ve|\log{\ve}|/4)$, we have
\begin{equation}
c_1\ve^{5/4}\leq u_{\ve}(r)\leq c_2\ve,
\end{equation}
for some constants $c_1>0$ and $c_2>0$,
and hence
\begin{equation}
\begin{aligned}
&\frac{\ve^N}{D_0(\ve)r^{N-1}u^2_{\ve}(r)}\int_0^rs^{N-1}v_{\ve}(s)u^3_{\ve}(s)ds\\
&\leq \frac{\|v_{\ve}\|_{\infty}\ve^N}{D_0(\ve)r^{N-1}c_1^2\ve^{5/2}}\int_0^rc_2^3\ve^{3}s^{N-1}ds\\
&=\frac{c_2^3\|v_{\ve}\|_{\infty}\ve^{N+\frac{1}{2}}}{N c_1^2D_0(\ve)}r
=O\left(\frac{\ve^{N+\frac{1}{2}}}{D_0(\ve)}\right)r.
\end{aligned}
\end{equation}

\item For  $r\in [5\ve|\log{\ve}|/4, R]$, we have
\begin{equation}
c_1\ve^{2}\leq u_{\ve}(r)\leq c_2\ve^{5/4},
\end{equation}
for some constants $c_1>0$ and $c_2>0$,
and hence
\begin{equation}
\begin{aligned}
&\frac{\ve^N}{D_0(\ve)r^{N-1}u^2_{\ve}(r)}\int_0^rs^{N-1}v_{\ve}(s)u^3_{\ve}(s)ds\\
&\leq \frac{\|v_{\ve}\|_{\infty}\ve^N}{D_0(\ve)r^{N-1}c_1^2\ve^{5/2}}\int_0^rc_2^3\ve^{3}s^{N-1}ds\\
&=\frac{c_2^3\|v_{\ve}\|_{\infty}\ve^{N-\frac{1}{4}}}{N c_1^2D_0(\ve)}r
=O\left(\frac{\ve^{N-\frac{1}{4}}}{D_0(\ve)}\right)r.
\end{aligned}
\end{equation}
\end{itemize}
This finishes the proof of \eqref{dv3}. The estimate \eqref{var-v}  then follows from \eqref{dv3} immediately by integrating equation \eqref{eq-v} over the interval $[0, r]$, for any $r\in [0, R]$.
\end{proof}

As a consequence of Lemma \ref{lem-est-v}, we can obtain the approximate value of $v_{\ve}$. Indeed,
integrating the equation  \eqref{eq-v} over $B_R$ we obtain
\begin{equation}
\int_{B_R}(\gamma_0-\ve^{-N}v_{\ve}u_{\ve}^3)dx=0.
\end{equation}
Let us set
\begin{equation}
\beta(\ve):=\frac{1}{\sqrt{D_0(\ve)}}.
\end{equation}
Then, using \eqref{u2-}, we obtain
\begin{equation*}
\begin{aligned}
&\gamma_0|B_R|=\ve^{-N}\int_{B_R}v_{\ve}\left[\alpha_{0}\ve^N+[v_{\ve}(0)]^{-1/2}w\left(\frac{x}{\ve}\right)+\phi_{\ve}(x)\right]^3dx\\
&=[1+O(\beta^2)]\left([v_{\ve}(0)]^{-1/2}\ve^{-N}\int_{B_R}w^3\left(\frac{x}{\ve}\right)dx+I_{\ve}\right)\\
&=O(\ve^N)+\left[1+O\left(\beta^2\right)\right][v_{\ve}(0)]^{-1/2}\int_{B_{R/\ve}}w^3(y)dy\\
&=\left[1+O\left(\beta^2\right)\right][v_{\ve}(0)]^{-1/2}\int_{\R^N}w^3(y)dy+O(\ve^{N}),
\end{aligned}
\end{equation*}
where
\[
I_{\ve}=\int_{B_R}\left[v_{\ve}(0)\alpha_{0}^3\ve^{2N}+3\alpha_{0}^2\ve^{N}[v_{\ve}(0)]^{\frac{1}{2}}w\left(\frac{x}{\ve}\right)+3\alpha_{0}w^2\left(\frac{x}{\ve}\right)\right]dx
\]
\[
=O(1)\times\left(\ve^{3N}|B_R|+\ve^{2N}\int_{\R^N}w(y)dy+\ve^{N}\int_{\R^N}w^2(y)dy\right)\\
\]
\[
=O(\ve^N).
\]
As a consequence we have
\begin{equation}
v_{\ve}(0)=\frac{\left(\int_{\R^N}w^3(y)dy\right)^2}{\gamma_0^2|B_R|^2}+O\left(\frac{1}{D_0(\ve)}\right),
\end{equation}
and
\begin{equation}
v_{\ve}(r)=\frac{\left(\int_{\R^N}w^3(y)dy\right)^2}{\gamma_0^2|B_R|^2}+O\left(\frac{1}{D_0(\ve)}\right),
\end{equation}
for all $r\in [0, R]$.

For the convenience, in the rest of the paper, we denote
\begin{equation}\label{v0}
v_0:=\frac{\left(\int_{\R^N}w^3(y)dy\right)^2}{\gamma_0^2|B_R|^2}.
\end{equation}

\section{A nonlocal eigenvalue problem}\label{sec: NLEP}
As a first step to study the linear stability of spike steady states of \eqref{eqn0} as $\ve\to 0$, we derive a nonlocal linear eigenvalue problem (NLEP for short).  As is well-known that, for small $\ve>0$, the stability of the spike steady states of \eqref{eqn0} is determined by this NLEP. We would like to note that the methods in this section, as those in Section \ref{sec:redToSin}, only work for the radial case.

For a ball $B\subset\R^N$, we set
\[
H^1_r(B):=\{\phi\in H^1(B)|\phi(x)=\phi(|x|)\},\]
\[\qquad L^2_r(B):=\{\phi\in L^2(B)|\phi(x)=\phi(|x|)\}.
\]
Linearizing the system \eqref{eqn0} around the steady state $(u_{\ve}, v_{\ve})$ with
\begin{equation}
\left\{\begin{aligned}
&u(x, t)=u_{\ve}(r)+\phi_{\ve}(r)e^{\lambda_{\ve} t},\quad 0<r<R,\\
&v(x, t)=v_{\ve}(r)+\psi_{\ve}(r)e^{\lambda_{\ve} t},\quad 0<r<R,\\
&\nabla_r\phi_{\ve}(0)=\nabla_r\phi_{\ve}(R)=0,\\ &\nabla_r\psi_{\ve}(0)=\nabla_r\psi_{\ve}(R)=0,
\end{aligned}\right.
\end{equation}
where $\lambda_{\ve}\in\mathbb{C}$ is some complex number and $\phi_{\ve}\in H^1_r(B_{R}), \psi_{\ve}\in H^1_{r}(B_R)$,
we deduce the following eigenvalue problem
\begin{equation}\label{lin-eq}
\left\{
\begin{aligned}&\ve^2\Delta \phi_{\ve} -\phi_{\ve}+3v_{\ve}u_{\ve}^2\phi_{\ve}+u_{\ve}^3\psi_{\ve}
=\lambda_{\ve}\phi_{\ve},\\
&\frac{D_0(\ve)}{{\ve^{2N}}}\nabla\left[u_{\ve}^2\nabla\psi_{\ve}+2u_{\ve}\phi_{\ve} \nabla v_{\ve} \right]\\
&\qquad =\ve^{-N}(3v_{\ve}u_{\ve}^2\phi_{\ve}+u_{\ve}^3\psi_{\ve})
+\tau(\ve)\lambda_{\ve}[2u_{\ve}v_{\ve}\phi_{\ve}+u_{\ve}^2\psi_{\ve}],\\
&\nabla_r\phi_{\ve}(0)=\nabla_{r}\phi_{\ve}(R)=0,\\
&\nabla_r\psi_{\ve}(0)=\nabla_r\psi_{\ve}(R)=0.
\end{aligned}
\right.
\end{equation}

If $\ve^N\tau(\ve)Re(\lambda_{\ve})< -c$ for some (small) positive constant $c>0$, then $\lambda_{\ve}$ is a stable eigenvalue. So we only consider the case  $\ve^N\tau(\ve)Re(\lambda_{\ve})\geq -c$. We shall derive the limiting eigenvalue problem as $\ve\to 0$, which turns out to be a nonlocal eigenvalue problem (NLEP). Using the method in \cite{Dancer2001} and similar arguments in Sections \ref{sec: NLEP1} and \ref{sec: NLEP2}, it is not difficult to show that the set $\{\lambda\in \mathbb{C}|\lambda_{\ve}\geq -c\ve^{-N}[\tau(\ve)]^{-1}\}$ is bounded. Hence we can assume, up to a subsequence, $\lambda_{\ve}\to \lambda\in \mathcal{C}$.

From the second equation of \eqref{lin-eq} we  obtain
\begin{equation}\label{psi1}
\begin{aligned}
\nabla_r&\psi_{\ve}(r)=-\frac{2\phi_{\ve}(r)\nabla_rv_{\ve}(r)}{u_{\ve}(r)}\\
&+\frac{\ve^N}{D_0(\ve)r^{N-1}u_{\ve}^2(r)}\int_0^rs^{N-1}[3v_{\ve}(s)u_{\ve}^2(s)\phi_{\ve}(s)+u_{\ve}^3(s)\psi_{\ve}(s)]ds\\
&+\frac{\ve^{2N}\tau(\ve)\lambda_{\ve}}{D_0(\ve)r^{N-1}u_{\ve}^2(r)}\int_0^rs^{N-1}[2u_{\ve}(s)v_{\ve}(s)\phi_{\ve}(s)+u_{\ve}^2(s)\psi_{\ve}(s)]ds.
\end{aligned}
\end{equation}
We may assume that there exists some constant $C>0$ such that
\begin{equation}
|\phi_{\ve}(r)|\leq Cw\left(\frac{r}{\ve}\right),\qquad |\psi_{\ve}(r)|\leq C,\qquad\text{for all}\;\;r\in [0, R].
\end{equation}
Thus the first term in the right-hand side of \eqref{psi1} has the estimate
\begin{equation}
\left|-\frac{2\phi_{\ve}(r)\nabla_rv_{\ve}(r)}{u_{\ve}(r)}\right|\leq C\left|\nabla_rv_{\ve}(r)\right|=O\left(\frac{1}{D_0(\ve)}\right)r.
\end{equation}
The second term in the right-hand side of \eqref{psi1} can be estimated as in the estimation of $\nabla_rv_{\ve}(r)$:
\begin{equation}
\begin{aligned}
\frac{\ve^N}{D_0(\ve)r^{N-1}u_{\ve}^2(r)}&\int_0^rs^{N-1}[3v_{\ve}(s)u_{\ve}^2(s)\phi_{\ve}(s)+u_{\ve}^3(s)\psi_{\ve}(s)]ds\\
&\leq \frac{C\ve^N}{D_0(\ve)r^{N-1}u_{\ve}^2(r)}\int_0^rs^{N-1}v_{\ve}(s)u_{\ve}^3(s)dx\\
&=O\left(\frac{1}{D_0(\ve)}\right)r.
\end{aligned}
\end{equation}
For the third term in the right-hand side of \eqref{psi1} we have
\begin{equation}
\begin{aligned}
&\left|\frac{\ve^{2N}\tau(\ve)\lambda_{\ve}}{D_0(\ve)r^{N-1}u_{\ve}^2(r)}\int_0^rs^{N-1}[2u_{\ve}(s)v_{\ve}(s)\phi_{\ve}(s)+u_{\ve}^2(s)\psi_{\ve}(s)]ds\right|\\
&\leq C\frac{\ve^{2N}\tau(\ve)|\lambda_{\ve}|}{D_0(\ve)u_{\ve}^2(r)}\int_0^ru_{\ve}^2(s)ds\\
&\leq C \frac{\ve^{2N}\tau(\ve)|\lambda_{\ve}|}{D_0(\ve)u_{\ve}^2(r)}\left(\int_0^{\ve|\log{\ve}|}+\int_{\ve|\log{\ve}|}^r\right)u_{\ve}^2(s)ds\\
&=I_1+I_2.
\end{aligned}
\end{equation}
On the interval $[0, \ve|\log{\ve}|)$ we have $c_1\ve\leq u_{\ve}\leq c_2$ and hence
\[\begin{aligned}
I_1&= C \frac{\ve^{2N}\tau(\ve)|\lambda_{\ve}|}{D_0(\ve)u_{\ve}^2(r)}\int_0^{\ve|\log{\ve}|}u_{\ve}^2 ds\\
&\leq C\frac{\ve^{2N-2}\tau(\ve)|\lambda_{\ve}|}{D_0(\ve)}\int_0^{\ve|\log{\ve}|}ds\\
&=O\left(\frac{1}{D_0(\ve)}\right)\times \ve^{2N-1}|\log{\ve}|\tau(\ve)|\lambda_{\ve}|.
\end{aligned}\]
On the interval $[\ve|\log{\ve}|, r)$ we have $c_1\ve^2\leq u_{\ve}\leq c_2\ve$ and hence
\[\begin{aligned}
I_2&= C \frac{\ve^{2N}\tau(\ve)|\lambda_{\ve}|}{D_0(\ve)u_{\ve}^2(r)}\int_{\ve|\log{\ve}|}^ru_{\ve}^2(s)ds\\
&\leq C\frac{\ve^{2N}\tau(\ve)|\lambda_{\ve}|}{D_0(\ve)\ve^4}\int_{\ve|\log{\ve}|}^r\ve^2ds\\
&=O\left(\frac{1}{D_0(\ve)}\right)\times \ve^{2N-2}\tau(\ve)|\lambda_{\ve}|.
\end{aligned}\]

The above estimate on $\nabla_r\psi$ is very rough. In fact, if we divides the interval $[0, r]$ into finite sufficiently small subintervals and make estimates on these subintervals, we can obtain the following more refined estimate
\begin{equation}\label{dpsi}
\nabla_r\psi_{\ve}(r)=O\left(\frac{1}{D_0(\ve)}\right)\times \ve^{2N-\delta}\tau(\ve)|\lambda_{\ve}|r\quad\text{for any small}\;\;\delta>0.
\end{equation}
As  a consequence of \eqref{dpsi} we obtain
\begin{equation}\label{psi_osc}
\psi_{\ve}(r)-\psi_{\ve}(0)=O\left(\frac{1}{D_0(\ve)}\right)\times \ve^{2N-\delta}\tau(\ve)|\lambda_{\ve}|r^2\quad\text{for all}\;\;r\in [0, R].
\end{equation}
In view of \eqref{tau}, we have
\begin{equation}
\psi_{\ve}(r)-\psi_{\ve}(0)=O\left(\frac{\ve^{N-\delta}}{D_0(\ve)}\right)|\lambda_{\ve}|r^2\quad\text{for all}\;\;r\in [0, R].
\end{equation}
This estimate will enable us to derive an NLEP as a limiting problem of the eigenvalue problem \eqref{lin-eq}. As mentioned before, we can prove that the set $\{\lambda\in \mathbb{C}|\lambda_{\ve}\geq -c\ve^{-N}[\tau(\ve)]^{-1}\}$ is bounded. Hence we can assume, up to a subsequence, $\lambda_{\ve}\to \lambda\in \mathcal{C}$.

We rescale $\phi_{\ve}(x)=\hat{\phi}_{\ve}(y)$ with $x=\ve y$ and assume that
\begin{equation}
\|\hat{\phi}_{\ve}\|_{H^2(B_{R/\ve})}=1.
\end{equation}
By a standard procedure, we can extend the definition of $ \hat{\phi}_{\ve}(y)$ to the whole of $\R^N$, still radial, denoted by $\tilde{\phi}_{\ve}$, with $C^{-1}\leq\|\tilde{\phi}_{\ve}\|_{H^2(\R^N)}\leq C$. So, up to a subsequence, we can  assume  $\tilde{\phi}_{\ve}\to \tilde{\phi}$ in $H^1(\R^N)$, as $\ve\to 0$.

Integrating the second equation of \eqref{lin-eq} over $B_R$, taking the limit $\ve\to 0$, and taking note of the exponential decay of $w$,  we obtain as $\ve\to 0$ that
\begin{equation}\label{phi_psi_int}
\int_{\R^N}[3w^2\tilde{\phi}+v_0^{-3/2}w^3\psi(0)+\ve^{N}\tau(\ve)\lambda_{\ve}(2v_0^{1/2}w\tilde{\phi}+v_0^{-1}w^2\psi(0))]dy=o(1).
\end{equation}
where $v_0$ is defined by \eqref{v0}.

From \eqref{phi_psi_int} we obtain
\[
\begin{aligned}
\psi(0)
\sim&-\left[v_0^{-3/2}\int_{\R^N}w^3+\ve^N\tau(\ve)\lambda_{\ve} v_0^{-1}\int_{\R^N} w^2\right]^{-1}\\&\times\left[3\int_{\R^N}w^2\tilde{\phi}+2\ve^N\tau(\ve)\lambda_{\ve} v_0^{1/2}\int_{\R^N}w\tilde{\phi}\right]\\
&=-\left[\frac{\gamma_0^3|B_R|^3}{\left(\int_{\R^N}w^3\right)^2}+\frac{\gamma_0^2|B_R|^2\ve^N\tau(\ve)\lambda_{\ve} \int_{\R^N} w^2}{\left(\int_{\R^N}w^3\right)^2}\right]^{-1}\\&\times\left[3\int_{\R^N}w^2\tilde{\phi}+2\frac{\ve^N\tau(\ve)\lambda_{\ve}}{\gamma_0|B_R|} \int_{\R^N}w^3\int_{\R^N}w\tilde{\phi}\right].
\end{aligned}
\]

Letting $\ve\to 0$ in the first equation of \eqref{lin-eq} we obtain  the following  NLEP:
\begin{equation}\label{NLEP}
\begin{aligned}
\mathcal{L}\tilde{\phi}:=L_0\tilde{\phi}&
- \frac{3\int_{\R^N}w^2\tilde{\phi}\,dy} {\int_{\R^N}w^3\,dy\left[1+\frac{\ve^N\tau(\ve)\lambda}{\gamma_0|B_R|}\int_{\R^N}w^2\,dy\right]}w^3\\
&-\frac{2\ve^N\tau(\ve)\lambda\int w\tilde{\phi}\,dy}{\gamma_0|B_R|+\ve^N\tau(\ve)\lambda\int_{\R^N}w^2\,dy}w^3=\lambda\tilde{\phi},
\end{aligned}
\end{equation}
where $\tilde{\phi}\in H_r(\R^N)$ and
\begin{equation}
L_0\tilde{\phi}:=\Delta \tilde{\phi}-\tilde{\phi}+3w^2\tilde{\phi}.
\end{equation}
Put
\begin{equation}
\tilde{\tau}=\frac{\ve^N\tau(\ve)}{\gamma_0|B_R|}\int_{\R^N}w^2\,dy.
\end{equation}
Then the NLEP has the form
\begin{equation}\label{NLEP++0}
\mathcal{L}\tilde{\phi}:=L_0\tilde{\phi}
- \frac{3}{1+\tilde{\tau}\lambda}\frac{\int_{\R^N}w^2\tilde{\phi}\,dy} {\int_{\R^N}w^3\,dy}w^3
-\frac{2\tilde{\tau}\lambda}{1+\tilde{\tau}\lambda}\frac{\int_{\R^N} w\tilde{\phi}\,dy}{\int_{\R^N}w^2\,dy}w^3
=\lambda\tilde{\phi},
\end{equation}
where $\tilde{\phi}\in H_r^1(\R^N)$.
In particular, letting $\lambda=0$ in \eqref{NLEP++0}, we deduce that the  nonlocal linear problem
\begin{equation}\label{NL}
L_0\tilde{\phi}-3 w^3\frac{\int_{\R^N}w^2\tilde{\phi}\,dy}{\int_{\R^N}w^3\,dy}=0,\hskip 2cm \tilde{\phi}\in H_r^1(\R^N)
\end{equation}
has a nontrivial solution $\|\tilde{\phi}\|_{H^2(\R^N)}>0$. However, according to \cite{Wei1999, Wei2001}, \eqref{NL} has only the trivial solution $\tilde{\phi}\equiv 0$. The contradiction implies that the solution $(u_{\ve}, v_{\ve})$ we are trying to construct is nondegenerate.

We will give a more detailed discussion of the NLEP \eqref{NLEP++0} in Sections \ref{sec: NLEP1} and \ref{sec: NLEP2}.

\section{The existence of radial spike solutions}\label{sec: existence}

In this section, we prove the existence of  radial spike solutions of \eqref{eq} concentrating at the center of the ball  $B_R$ as $\ve\to 0$.  We divide the proof into two steps. First we construct  radial approximate solutions to \eqref{eq} which concentrate at the center of the ball $B_R$. Then we use the contraction mapping principle to show that there exists  exact spike solutions of \eqref{eq} as a small perturbation of the approximate solutions constructed in the first step.

\subsection{Approximate solutions}
Let $\chi:\R\to [0, 1]$ be a smooth cut-off function such that $\chi(s)=1$ for $|s|<1$ and $\chi(s)=0$ for $|s|>2$. Set $R_0=\frac{1}{3}R$ and
\begin{equation}
\tilde{w}_{\ve}(r)=w\left(\frac{r}{\ve}\right)\chi\left(\frac{r}{R_0}\right).
\end{equation}
It is easy to see that $\tilde{w}_{\ve}$ satisfies
\begin{equation}
 \ve^2\Delta \tilde{w}_{\ve}-\tilde{w}_{\ve}+\tilde{w}_{\ve}^3=e.s.t.
\end{equation}
in $L^2(B_R)$, where e.s.t denotes an exponentially small term.

Set
\begin{equation}
w_{\ve}=\frac{1}{\sqrt{v_{\ve}(0)}}\tilde{w}_{\ve},
\end{equation}
where $v_{\ve}=T[w_{\ve}]$
is defined by
\begin{equation}\label{eq-T}
\left\{\begin{aligned}
&\frac{D_0(\ve)}{\ve^{2N} }\nabla [(\alpha_{0}\ve^N+w_{\ve})^2 \nabla T[w_{\ve}])\\
&\hskip 2cm-\ve^{-N}T[w_{\ve}](\alpha_{0}\ve^N+w_{\ve})^3+\gamma_0=0, \;\;&&x\in B_R,\\ &\partial_{\nu}T[w_{\ve}]=0,\quad &&x\in \partial B_R.
\end{aligned}\right.
\end{equation}
We let $r=\ve \rho$ and find that for all $\rho\geq 0$
\begin{equation*}
\begin{aligned}
&T[w_{\ve}](\ve \rho)-T[w_{\ve}](0)\\
=&\frac{\ve^{2N}}{D_0(\ve)}\int_0^{\ve \rho}\frac{1}{s^{N-1}(\alpha_{0}\ve^N+w_{\ve})^2}\int_0^s\tau^{N-1}[\ve^{-N} v_{\ve}(\tau)(\alpha_{0}\ve^N+w_{\ve})^3-\gamma_0]d\tau ds\\
=&\frac{\ve^N}{D_0(\ve)}\int_0^{\ve \rho}\frac{1}{s^{N-1}(\alpha_{0}\ve^N+w_{\ve})^2}\int_0^s \tau^{N-1} v_{\ve}(\tau)(\alpha_{0}\ve^N+w_{\ve})^3d\tau ds\\
&\quad -\frac{\gamma_0\ve^{2N}}{D_0(\ve)}\int_0^{\ve \rho}\frac{1}{s^{N-1}(\alpha_{0}\ve^N+w_{\ve})^2}\int_0^s\tau^{N-1} d\tau ds.
\end{aligned}
\end{equation*}
Setting
\[
W_\ve(s):=\alpha_{0}\ve^N+[v_{\ve}(0)]^{-1/2}w(s),
\]
and using the inequalities $(a+b)^3\leq 2^{3-1}(a^3+b^3)$  for $a, b\geq 0$, we estimate  that
\[
\begin{aligned}
&\quad\frac{\ve^N}{D_0(\ve)}\int_0^{\ve \rho}\frac{1}{s^{N-1}(\alpha_{0}\ve^N+w_{\ve})^2}\int_0^s \tau^{N-1} v_{\ve}(\tau)(\alpha_{0}\ve^N+w_{\ve})^3d\tau ds\\
&=\frac{v_{\ve}(0)\ve^N}{D_0(\ve)}\int_0^{\ve \rho}\frac{1}{s^{N-1}\left(\alpha_{0}\ve^N+w_{\ve}\right)^2}\int_0^s \tau^{N-1} \left(\alpha_{0}\ve^N+w_{\ve}\right)^3d\tau ds+h.o.t.\\
&=\frac{v_{\ve}(0)\ve^{N+2}}{D_0(\ve)}\int_0^{ \rho}\frac{1}{s^{N-1}W_{\ve}^2(s)}\int_0^s \tau^{N-1} W_{\ve}^3(s)d\tau ds
+h.o.t.\\
&\leq \frac{v_{\ve}(0)\ve^{N+2}}{D_0(\ve)}\int_0^{ \rho}\frac{1}{s^{N-1}W_{\ve}^3(s)}\int_0^s4\alpha_{0}^3\ve^{3N}\tau^{N-1} d\tau ds\\
&+\frac{v_{\ve}(0)\ve^{N+2}}{D_0(\ve)}\int_0^{ \rho}\frac{1}{s^{N-1}W_{\ve}^3(s)}\int_0^s 4[v_\ve(0)]^{-3/2}\tau^{N-1} w^3(\tau)d\tau ds+h.o.t.\\
&\leq \frac{4\alpha_{0}^3v_{\ve}(0)\ve^{4N+2}}{ND_0(\ve)}\int_0^{ \rho}\frac{s}{\left(\alpha_{0}\ve^N+[v_{\ve}(0)]^{-1/2}w\left(s\right)\right)^2}ds\\
&+\frac{v_{\ve}(0)\ve^{N+2}}{D_0(\ve)}\int_0^{ \rho}\frac{1}{\left(\alpha_{0}\ve^N+[v_{\ve}(0)]^{-1/2}w\left(s\right)\right)^2}\int_0^s 4[v_\ve(0)]^{-3/2} w^3(\tau)d\tau ds+h.o.t.\\
&=O(1)\times\left( \frac{\ve^{2N+2}}{D_0(\ve)}\rho^2
+\frac{\ve^{N+2}\rho}{D_0(\ve)}\int_{0}^{\rho}\frac{1}{\left(\ve^N+w\left(s\right)\right)^2}ds \right),
\end{aligned}
\]
and
\[\begin{aligned}
&\frac{\gamma_0\ve^{2N}}{D_0(\ve)}\int_0^{\ve \rho}\frac{1}{s^{N-1}(\alpha_{0}\ve^N+w_{\ve})^2}\int_0^s\tau^{N-1} d\tau ds\\
&=\frac{\gamma_0\ve^{2N}}{ND_0(\ve)}\int_0^{\ve \rho}\frac{s}{(\alpha_{0}\ve^N+w_{\ve})^2}ds\\
&=\frac{\gamma_0\ve^{2N+2}}{ND_0(\ve)}\int_0^{ \rho}\frac{s}{(\alpha_{0}\ve^N+[v_{\ve}(0)]^{-1/2}w(s))^2}ds\\
&=O(1)\times \frac{\ve^{2N+2}\rho}{D_0(\ve)}\int_0^{\rho}\frac{1}{(\ve^N+w(s))^2}ds.
\end{aligned}\]
For the estimate of the integral $\int_0^{\rho}\frac{1}{(\ve^N+w(s))^2}ds$, we have the following three different ways. Using the inequality $(\ve^N+w(s))^2\geq w^2(s)$, we have
\begin{equation}
\int_0^{\rho}\frac{1}{(\ve^N+w(s))^2}ds\leq O(1)\times\int_0^{\rho}s^{N-1}e^{2s}ds=O(1)\times \rho^{N-1} e^{2\rho}.
\end{equation}
Using the inequality $(\ve^N+w(s))^2\geq 2\ve^Nw(s)$, we have
\begin{equation}
\int_0^{\rho}\frac{1}{(\ve^N+w(s))^2}ds\leq O(\ve^{-N})\times\int_0^{\rho}s^{\frac{N-1}{2}}e^{s}ds=O(\ve^{-N})\times \rho^{\frac{N-1}{2}} e^{\rho}.
\end{equation}
Using the inequality $(\ve^N+w(s))^2\geq \ve^{2N}$, we have
\begin{equation}
\int_0^{\rho}\frac{1}{(\ve^N+w(s))^2}ds=O(\ve^{-2N})\times \rho.
\end{equation}

Hence we have the estimate
\begin{equation}
T[w_{\ve}](\ve \rho)-T[w_{\ve}](0)=O\left(\frac{1}{D_0(\ve)}\right)\times\left\{\begin{aligned} &\ve^{N+2}\rho^{N}e^{2\rho},\;\;&&\text{or}\\ &\ve^2\rho^{\frac{N+1}{2}}e^{\rho},\;\;&&\text{or} \\ &\ve^{2-N}\rho^2.\end{aligned}\right.\end{equation}

Therefore we have the following estimates that  for all $\rho\in [0, R/\ve]$:
\begin{equation}\label{T-var+}
\begin{aligned}
&|(T[w_{\ve}](\ve \rho)-T[w_{\ve}](0))|w_{\ve}^3\leq C\ve^{N+2} \rho^{\frac{3-N}{2}} e^{-\rho},\\
&|(T[w_{\ve}](\ve \rho)-T[w_{\ve}](0))|\ve^N w_{\ve}^2\leq C\ve^{N+2}\rho^{\frac{3-N}{2}}e^{-\rho},\\
&|(T[w_{\ve}](\ve \rho)-T[w_{\ve}](0))|\ve^{2N} w_{\ve}\leq C\ve^{N+2} \rho^{\frac{5-N}{2}}e^{-\rho}.
\end{aligned}
\end{equation}

Now if we define the norm
\begin{equation}
\|f\|_{**}=\|f\|_{L^2(B_{R/\ve})}+\sup_{0<\rho<R/\ve}[\max(e^{-\frac{1}{2}\rho}, \sqrt{\ve})]^{-1}|f(\rho)|,
\end{equation}
they by the decay of $w_{\ve}$ and the definition of the norm, we infer that
\begin{equation}\label{E3}
\|(T[w_{\ve}](\ve \rho)-T[w_{\ve}](0))w_{\ve}^3\|_{**}=O(\ve^{N+3/2}),
\end{equation}
\begin{equation}
\|(T[w_{\ve}](\ve \rho)-T[w_{\ve}](0))\ve^N w_{\ve}^2\|_{**}=O(\ve^{N+3/2}),
\end{equation}
\begin{equation}
\|(T[w_{\ve}](\ve \rho)-T[w_{\ve}](0))\ve^{2N}w_{\ve}\|_{**}=O(\ve^{N+3/2}).
\end{equation}

Let us now define
\begin{equation}
S_{\ve}[w_{\ve}]:=\ve^2\Delta w_{\ve}-w_{\ve}+T[w_{\ve}](\alpha_{0}\ve^N+w_{\ve})^3,
\end{equation}
where $T[w_{\ve}]$ is defined in \eqref{eq-v}. Then
\begin{equation}
\begin{aligned}
S_{\ve}[w_{\ve}]&=\ve^2\Delta w_{\ve}-w_{\ve}+T[w_{\ve}](\alpha_{0}\ve^N+w_{\ve})^3\\
&=\ve^2\Delta w_{\ve}-w_{\ve}+T[w_{\ve}](0)w^3_{\ve}\\
&+T[w_{\ve}](\alpha_{0}^3\ve^{3N}+3\alpha_{0}^2\ve^{2N}w_{\ve}+3\alpha_{0} \ve^Nw^2_{\ve})\\
&+(T[w_{\ve}](\ve \rho)-T[w_{\ve}](0))w_{\ve}^3\\
&=:E_1+E_2+E_3.
\end{aligned}
\end{equation}
We have
\[
E_1=\frac{1}{v_{\ve}(0)}(\Delta_y \tilde{w}_{\ve}-\tilde{w}_{\ve}+\tilde{w}_{\ve}^3)=e.s.t.,
\]
and
\[
E_2=O(\ve^N)\quad\text{in}\;\;L^2(B_{R/\ve}),
\]
since $T[w_{\ve}]$ is bounded in $L^\infty(B_{R/\ve})$ and $w_{\ve}$ is bounded in $L^2(B_{R/\ve})$.
\[
E_3=[v_{\ve}(0)]^{-3/2}(T[w_{\ve}](\ve \rho)-T[w_{\ve}](0))\tilde{w}_{\ve}^3=O(\ve^{N+2})
\]
in $L^2(B_{R/\ve})$ by the first inequality of \eqref{T-var+}.

Combining these estimates we conclude that
\begin{equation}\label{S}
\|S_{\ve}[w_{\ve}]\|_{**}=O(\ve^{N}).
\end{equation}

The estimate \eqref{S} shows that our choice of approximate solutions is suitable. This will enable us  to rigorously construct a steady state which is very close to the approximate solution.
\subsection{The existence of exact solutions}

In this subsection, we use the contraction mapping principle to prove the existence of a spike solution close to the approximate solution. To this end, we need to study the linearized operator
\[
L_{\ve}:H_r^2(B_{R/\ve})\to L_r^2(B_{R/\ve})
\]
given by
\begin{equation*}
L_{\ve}\phi:=S'_{\ve}[w_{\ve}]\phi=\Delta \phi-\phi+3T[w_{\ve}](\alpha_{0}\ve^N+w_{\ve})^2\phi+(\alpha_{0}\ve^N+w_{\ve})^3T'[w_{\ve}]\phi,
\end{equation*}
where for a given function $\phi\in H_r^2(B_{R/\ve})$ we define $T'[w_{\ve}]\phi$ to be the unique solution of
\begin{equation*}
\left\{
\begin{aligned}
&\nabla_y\left[(\alpha_{0}\ve^N+w_{\ve})^2(T'[w_{\ve}]\nabla_y\phi)\right]\\
& -\frac{\ve^{N+2}}{D_0(\ve)}\left(3T[w_{\ve}](\alpha_{0}\ve^N+w_{\ve})^2\phi+(\alpha_{0}\ve^N+w_{\ve})^3T'[w_{\ve}]\phi\right)=0,\;\;\text{in}\;\;B_{R/\ve},\\
&\partial_r(T'[w_{\ve}])=0,\quad\text{on}\;\;\partial B_{R/\ve}.
\end{aligned}
\right.
\end{equation*}

The norm of the error function $\phi$ is defined as
\begin{equation}
\|\phi\|_{*}=\|\phi\|_{H^2(B_{R/\ve})}+\sup_{\rho\in [0, R/\ve]}[\max(e^{-\rho/2}, \sqrt{\ve})]^{-1}|\phi(\rho)|.
\end{equation}

We recall the nonlocal linear problem \eqref{NL}:
\begin{equation}\label{NL+}
\mathcal{L}\phi=\Delta \phi-\phi+3w^2\phi-3w^3\frac{\int_{\R^N}w^2\phi\,dy}{\int_{\R^N}w^3\,dy}=0,\qquad \phi\in H_r^2(\R^N).
\end{equation}
By \cite{Wei1999} we know that
\[
\mathcal{L}:H^2_r(\R^N)\to L^2_r(\R^N)
\]
 is invertible and its inverse is bounded.

We will show that $L_{\ve}$ is a small perturbation of $\mathcal{L}$  in that $L_{\ve}$ is also invertible with a uniformly bounded inverse for sufficiently small $\ve>0$. This statement is contained in the following proposition.
\begin{prop}\label{L-invertible}
There exist positive constants $\ve_1$ and $\delta_1$ such that for all $\ve\in (0, \ve_1)$, there holds
\begin{equation}\label{coercive}
\|L_{\ve}\phi\|_{**}\geq \delta_1\|\phi\|_{*}.
\end{equation}
Moreover, the map
 \[
L_{\ve}:H_r^2(B_{R/\ve})\to L_r^2(B_{R/\ve})
\]
is surjective.
\end{prop}

\begin{proof}
Suppose that \eqref{coercive} is false. Then there exist sequences $\{\ve_k\}$ and $\{\phi_k\}$ with $\ve_k\to 0$ and $\phi_k=\phi_{\ve_k}$ such that
\begin{equation}
\|\phi_k\|_*=1\;\;\text{for}\;\; k=1, 2, \cdots,
\end{equation}
and
\begin{equation}
\|L_{\ve_k}\phi_k\|_{**}\to 0\;\;\text{as}\;\;k\to \infty.
\end{equation}
We define
\begin{equation}
\tilde{\phi}_{\ve}(y)=\phi(y)\chi\left(\frac{\ve |y|}{R_0}\right)
\end{equation}
with $R_0=R/3$.

By a standard procedure $\tilde{\phi}_{k}$ can be extended to be defined on $\R^N$ such that their norm in $H^2(\R^N)$ is still bounded by  a constant independent of $\ve$ for $\ve$ small enough. In the following we will deal with this extension.  Since $\{\tilde{\phi}_k\}$ is bounded in $H^2_{loc}(\R^N)$ it converges weakly to a limit $\tilde{\phi}$ in $H^2_{loc}(\R^N)$, and also strongly in $L^2_{loc}(\R^N)$ and $L^\infty_{loc}(\R^N)$. Then $\tilde{\phi}$ solves the equation $\mathcal{L}\tilde{\phi}=0$, which implies that $\tilde{\phi}=0$. By elliptic regularity we have $\|\tilde{\phi}_k\|_{H^2(\R^N)}\to 0$ as $k\to \infty$, which implies that $\|\phi_k\|_{H^2(B_{R/\ve})}\to 0$ as $k\to\infty$. The maximum principle then implies that $\|\phi_k\|_{*}\to 0$ as $k\to\infty$. This contradicts the assumption that $\|\phi_k\|_*=1$.

To complete the proof of the proposition, we need to show that the conjugate operator of $L_{\ve}$ (denoted by $L_{\ve}^*$) is injective from $H^2(B_{R/\ve})$ to $L^2(B_{R/\ve})$. The injectivity of $L^*_{\ve}$ is essentially the nondegeneracy condition we discussed in the end of Section \ref{sec: NLEP} and therefore omitted here.
\end{proof}

Now we are in a position to solve the equation
\begin{equation}
S_{\ve}[w_{\ve}+\phi]=0.
\end{equation}
Since $L_{\ve}$ is invertible (with its inverse $L^{-1}_{\ve}$), we can rewrite this equation as
\begin{equation}\label{FixedPoint}
\phi=-(L_{\ve}^{-1}\circ S_{\ve}[w_{\ve}])-(L^{-1}_{\ve}\circ N_{\ve}[\phi])\equiv M_{\ve}[\phi],
\end{equation}
where
\begin{equation}
N_{\ve}[\phi]=S_{\ve}[w_{\ve}+\phi]-S_{\ve}[w_{\ve}]-S'_{\ve}[w_\ve]\phi,
\end{equation}
and the operator $M_{\ve}$ is defined for $\phi\in H^2(B_{R/\ve})$. We will show that the operator $M_{\ve}$ is a contraction on
\[
\mathcal{B}_{\ve, \delta}\equiv \{\phi\in H^2(B_{R/\ve}):\|\phi\|_*<\delta\}
\]
if $\ve$ is small enough and $\delta$ is suitably chosen.

By \eqref{S} and Proposition \ref{L-invertible} we have that
\[
\begin{aligned}
\|M_{\ve}[\phi]\|_{*}\leq \delta_1^{-1}\left(\|S_{\ve}[w_{\ve}]\|_{**}+\|N_{\ve}[\phi]\|_{**}\right)\leq \delta_1^{-1}C_0(\ve^N+c(\delta)\delta),
\end{aligned}
\]
where $\delta_1>0$ is independent of $\delta>0, \ve>0$, and $c(\delta)\to 0$ as $\delta\to 0$. Similarly we can show that
\[
\begin{aligned}
\|M_{\ve}[\phi]-M_{\ve}[\tilde{\phi}]\|_{*}\leq \delta_1^{-1}C_0(\ve^N+c(\delta)\delta)\|\phi-\tilde{\phi}\|_*,
\end{aligned}
\]
where $c(\delta)\to 0$ as $\delta\to 0$. Choosing $\delta=C_1\ve^N$ for $\delta_1^{-1}C_0<C_1$ and taking $\ve$ small enough, then $M_{\ve}$ maps $\mathcal{B}_{\ve, \delta}$ into $\mathcal{B}_{\ve, \delta}$, so that it is a contraction mapping in $\mathcal{B}_{\ve, \delta}$. The existence of a fixed point $\phi_{\ve}$ now follows from the standard contraction mapping principle and $\phi_{\ve}$ is a solution of \eqref{FixedPoint}.

We have thus proved the following.
\begin{thm}
There exists $\ve_0>0, C_1>0$ such that for every $\ve\in (0, \ve_0)$, there is a unique $\phi_{\ve}\in H^2(B_{R/\ve})$ satisfying
\[
S[w_{\ve}+\phi_{\ve}]=0\quad\text{with}\;\;\|\phi_{\ve}\|_*\leq C_1\ve^{N}.
\]
\end{thm}

\section{The stability of the spike solutions in the two-dimensional case}\label{sec: NLEP1}

In this section, we consider the linear stability of the spike solution we obtained from the precious sections.  For $\ve$ small enough, it is sufficient to study the spectrum of the  NLEP: \begin{equation}\label{NLEP++}
\begin{aligned}
\mathcal{L}\phi:=L_0\phi&
- \frac{3}{1+\tilde{\tau}\lambda}\frac{\int_{\R^N}w^2\phi\,dy} {\int_{\R^N}w^3\,dy}w^3
-\frac{2\tilde{\tau}\lambda}{1+\tilde{\tau}\lambda}\frac{\int_{\R^N} w\phi\,dy}{\int_{\R^N}w^2\,dy}w^3\\
&=\lambda\phi,\hskip 2cm \phi\in H_r^1(\R^2),
\end{aligned}
\end{equation}
where
\begin{equation}
L_0\phi:=\Delta \phi-\phi+3w^2\phi,
\end{equation}
and
\begin{equation}
\tilde{\tau}=\frac{\ve^2\tau(\ve)}{\gamma_0|B_R|}\int_{\R^2}w^2\,dy.
\end{equation}

We begin our discussion of the NLEP by citing a result from \cite{Wei2001}.

\begin{lem}[Theorem 1 of \cite{Wei2001}]\label{lem:tau=0}
Suppose $\gamma>2$. There exists a constant $c_0>0$ such that all the eigenvalues of the eigenvalue problem
\begin{equation}
\begin{aligned}
&L_0\phi-\gamma\frac{\int_{\R^2}w^2\phi}{\int_{\R^2} w^3}w^3=\lambda\phi,\\
&\lambda\in \mathbb{C}, \qquad \phi\in H^1(\R^2),
\end{aligned}
\end{equation}
satisfies $Re(\lambda)\leq -c_0$.
\end{lem}

Consider the following nonlocal eigenvalue problem
\begin{equation}
\label{s1}
\left\{\begin{aligned}
&L_0\phi -\frac{3}{1+\tilde{\tau} \lambda} \frac{\int_{\R^2} w^2 \phi}{\int_{\R^2} w^3} w^3-\frac{2\tilde{\tau} \lambda}{1+\tilde{\tau} \lambda} \frac{\int_{\R^2} w \phi}{ \int_{\R^2} w^2} w^3 =\lambda \phi,\\
&\lambda\in\mathbb{C},\qquad\phi\in H^1(\R^2),
\end{aligned}\right.
\end{equation}
where
\[
L_0\phi=\Delta\phi-\phi+3w^2\phi.
\]

If $\tilde{\tau}=0$, by the remark of Theorem 1 of \cite{Wei2001}, any eigenvalue must satisfy $ Re (\lambda)<0$.  Hence for $\tilde{\tau} $ small problem (\ref{s1}) is stable.

For large $\tilde{\tau}$ we have the following result.

\begin{thm} Let $(\tilde{\tau}, \lambda)$ be a pair satisfying \eqref{s1} with a nontrivial eigenfunction $\phi$. Then
\begin{itemize}
\item [(i)]
There exists $\tilde{\tau}_0>0$ such that for $\tilde{\tau} >\tilde{\tau}_0$, any eigenvalue of (\ref{s1}) with $ Re (\lambda)\geq 0$ must be of the order $ c_0^{1/3} \tilde{\tau}^{-1/3} e^{ \pm \frac{\pi}{3} i}$ with $c_0= \frac{\int_{\R^2} w^2}{2\int_{\R^2} (w_0)^2}>0$. Conversely for $\tilde{\tau}$ large there exist a pair of eigenvalues on the right half plane with $\lambda \sim c_0^{1/3} \tilde{\tau}^{-1/3} e^{ \pm \frac{\pi}{3} i}$.

\item  [(ii)] There exists a Hopf bifurcation at some $\tilde{\tau}_h>0$.
\end{itemize}
\end{thm}

We prove the result with a series of claims.

From now on we may assume that $ \lambda=\lambda_R+i \lambda_I$ with $\lambda_R\in \R, \lambda_I\in \R$ and $\lambda_R \geq 0$.

\medskip

\noindent
{\bf Claim 1:} If $\lambda_R \geq 0$, then $ |\lambda|\leq C$ where $C$ is independent of $\tilde{\tau}$.

\begin{proof}

Multiplying (\ref{s1}) by $\bar{\phi}$ (the complex conjugate of $\phi$), and integrating over $\R^2$ we obtain
\begin{equation}\label{1+}
\begin{aligned}
 &\int_{\R^2} \left(|\nabla \phi|^2 +|\phi|^2 - 3 w^2 |\phi|^2 + \lambda \int_{\R^2} |\phi|^2\right) \\
 &= -\frac{3}{1+\tilde{\tau} \lambda} \frac{\int_{\R^2} w^2 \phi}{\int_{\R^2} w^3 } \int_{\R^2} w^3\bar{\phi}-\frac{2\tilde{\tau} \lambda}{1+\tilde{\tau} \lambda} \frac{\int_{\R^2} w \phi}{ \int_{\R^2} w^2} \int_{\R^2} w^3\bar{\phi}
\end{aligned}
\end{equation}
Let $\mu_0>0$ be the first eigenvalue of $L_0$ given in Lemma \ref{lem:L0}. We have by the  variational representation of $\mu_0$ that   $$\int_{\R^2} (|\nabla \phi|^2 +|\phi|^2 - 3 w^2 |\phi|^2)\geq -\mu_0 \int_{\R^2} |\phi|^2.$$
The integrals on the right-hand side of \eqref{1+} can be estimated using the Holder inequalities:
\[
\left|\int_{\R^2} w^{k}\phi\right|\leq \left(\int_{\R^2}w^{2k}\right)^{1/2}\left(\int_{\R^2}|\phi|^2\right)^{1/2},\;\;k=1, 2,\]
and
\[\left|\int_{\R^2} w^3\bar{\phi}\right|\leq \left(\int_{\R^2}w^2\right)^{1/2}\left(\int_{\R^2}|\phi|^2\right)^{1/2}.
\]
It follows that
$$ |\lambda| \leq \mu_0  + C \left(\Big|\frac{3}{1+\tilde{\tau} \lambda}\Big|+ \Big|\frac{2\tilde{\tau} \lambda}{1+\tilde{\tau} \lambda}\Big|\right),$$
which together with $\lambda_R\geq 0$ implies $|\lambda|\leq C$.

The claim is proved.
\end{proof}

\medskip

\noindent
{\bf Claim 2:} If $\tilde{\tau} \to +\infty$  then $ \lambda \to 0$.

\begin{proof}
Suppose the claim is not true. Then up to a subsequence, we have $\tilde{\tau}\to +\infty$ $\lambda \to \lambda_\infty \not = 0$. Then $\frac{1}{1+\tilde{\tau} \lambda} \to 0$, $\frac{\tilde{\tau} \lambda}{1+\tilde{\tau} \lambda} \to 1$, and we obtain the following limiting problem
\begin{equation}
\label{2}
\Delta \phi -\phi+ 3w^2 \phi -2 \frac{\int_{\R^2} w \phi}{ \int_{\R^2} w^2} w^3 =\lambda_\infty \phi.
\end{equation}

The rest of the proof follows the line of the proof of Case 2 of Theorem 1.4 of \cite{Wei1999} (for the case $p=3=1+\frac{4}{2}$), with a few  necessary modifications.

Let the linear operator $L_1:H^1(\R^2)\to L^2(\R^2)$ be defined by
\begin{equation*}
L_1\phi:=L_0\phi-2\frac{\int_{\R^2} w\phi}{\int_{\R^2} w^2}w^3-2\frac{\int_{\R^2} w^3\phi}{\int_{\R^2} w^2}w+2\frac{\int_{\R^2} w^{4}\int_{\R^2} w\phi}{\left(\int_{\R^2} w^2\right)^2}w,\qquad \phi\in H^1(\R^2).
\end{equation*}
Then $L_1$ is self-adjoint.

According to Lemma 5.2 of \cite{Wei1999}, we have
\begin{itemize}
\item The kernel of $L_1$ is given by $X_1=\textnormal{span}\{w, w_0, \frac{\partial w}{\partial y_j}, j=1, 2\}$, where
    \[
   w_0:=\frac{1}{2}w+\frac{1}{2}y\nabla w.
    \]
    \item There exists a positive constant $a_1>0$ such that for all $\phi\in H^1(\R^2)$
    \[\begin{aligned}
    L_1(\phi, \phi)=L_0(\phi, \phi)&
    +4\frac{\int_{\R^2} w\phi\int_{\R^2} w^3\phi}{\int_{\R^2} w^2}-2\frac{\int_{\R^2} w^{4}}{\left(\int_{\R^2} w^2\right)^2}\left(\int_{\R^2} w\phi\right)^2\\
    &\geq a_1d^2_{L^2(\R^2)}(\phi, X_1),
    \end{aligned}\]
    where
    \[
    L_0(\phi, \phi):=\int_{\R^2}(|\nabla \phi|^2+\phi^2-3w^2\phi^2),
    \]
    and
    $d_{L^2(\R^2)}(\phi, X_1)$ is the distance of $\phi$ to $X_1$ in the space of $L^2(\R^2)$.
\end{itemize}

Now we are ready to prove the claim. Let $\lambda_\infty=\lambda_R+i\lambda_I$ and $\phi=\phi_R+i\phi_I$. Then we have the system of equations
\begin{equation}\label{RP}
L_0\phi_R-2 \frac{\int_{\R^2} w \phi_R}{ \int_{\R^2} w^2} w^3 =\lambda_R \phi_R-\lambda_I\phi_I,
\end{equation}
\begin{equation}\label{IP}
L_0\phi_I-2 \frac{\int_{\R^2} w \phi_I}{\int_{\R^2} w^2} w^3 =\lambda_R \phi_I+\lambda_I\phi_R.
\end{equation}
Multiplying \eqref{RP} by $\phi_R$, \eqref{IP} by $\phi_I$, integrating over $\R^2$, and summing up, we obtain
\begin{equation}
\begin{aligned}
L_0(\phi_R, \phi_R)&+L_0(\phi_I, \phi_I)+\lambda_R\int_{\R^2} (\phi_R^2+\phi_I^2)\\
&=-\frac{2}{\int_{\R^2} w^2}\left(\int_{\R^2} w^3\phi_R\int_{\R^2} w\phi_R+\int_{\R^2} w^3\phi_I\int_{\R^2} w\phi_I\right),
\end{aligned}
\end{equation}
or in the form
\begin{equation}\label{RI}
\begin{aligned}
&L_1(\phi_R, \phi_R)+L_1(\phi_I, \phi_I)+\lambda_R\int_{\R^2} (\phi_R^2+\phi_I^2)\\
&=2\frac{\int_{\R^2} w\phi_R\int_{\R^2} w^3\phi_R+\int_{\R^2} w\phi_I\int_{\R^2} w^3\phi_I}{\int_{\R^2} w^2}\\
&-\frac{2\int_{\R^2} w^{4}}{\left(\int_{\R^2} w^2\right)^2}\left[\left(\int_{\R^2} w\phi_R\right)^2+\left(\int_{\R^2} w\phi_I\right)^2\right].
\end{aligned}
\end{equation}
Multiplying \eqref{RP} and \eqref{IP} by $w$ respectively, and integrating over $\R^2$, we obtain
\begin{equation}\label{RP+}
2\int_{\R^2} w^3\phi_R-2 \frac{\int_{\R^2} w^4}{ \int_{\R^2} w^2}\int_{\R^2} w \phi_R =\lambda_R\int_{\R^2}  w\phi_R-\lambda_I\int_{\R^2} w\phi_I,
\end{equation}
\begin{equation}\label{IP+}
2\int_{\R^2} w^3\phi_I-2 \frac{\int_{\R^2} w^4}{ \int_{\R^2} w^2}\int_{\R^2} w \phi_I =\lambda_R\int_{\R^2}  w\phi_I+\lambda_I\int_{\R^2} w\phi_R.
\end{equation}
Multiplying \eqref{RP+} by $\int_{\R^2} w\phi_R$, \eqref{IP+} by $\int_{\R^2} w\phi_I$, and summing up, we obtain
\begin{equation}\label{RI+}
\begin{aligned}
&\int_{\R^2} w\phi_R\int_{\R^2} w^3\phi_R+\int_{\R^2} w\phi_I\int_{\R^2} w^3\phi_I\\
&=\left(\frac{\lambda_R}{2}+\frac{\int_{\R^2} w^4}{ \int_{\R^2} w^2}\right)\left[\left(\int_{\R^2} w\phi_R\right)^2+\left(\int_{\R^2} w\phi_I\right)^2\right].
\end{aligned}
\end{equation}
Plugging \eqref{RI+} into \eqref{RI} we obtain
\begin{equation}\label{RI++}
\begin{aligned}
&L_1(\phi_R, \phi_R)+L_1(\phi_I, \phi_I)+\lambda_R\int_{\R^2} (\phi_R^2+\phi_I^2)\\
&
=\frac{\lambda_R}{\int_{\R^2} w^2}\left[\left(\int_{\R^2} w\phi_R\right)^2+\left(\int_{\R^2} w\phi_I\right)^2\right].
\end{aligned}
\end{equation}
We decompose
\[
\phi_R=b_Rw+c_Rw_0+\sum_{j=1}^2d_{Rj}\frac{\partial w}{\partial y_j}+\phi_R^{\perp},\quad\phi_R^{\perp}\perp X_1,
\]
\[
\phi_I=b_Iw+c_Iw_0+\sum_{j=1}^2d_{Ij}\frac{\partial w}{\partial y_j}+\phi_I^{\perp},\quad\phi_I^{\perp}\perp X_1,
\]
and  put them into \eqref{RI++} and calculate
\[
\left(\int_{\R^2} w\phi_R\right)^2+\left(\int_{\R^2} w\phi_I\right)^2=(b_R^2+b_I^2)\left(\int_{\R^2}w^2\right)^2,
\]
\[\begin{aligned}
\int_{\R^2}(\phi_R^2+\phi_I^2)=(b_R^2+b_I^2)\int_{\R^2}w^2+\int_{\R^2}[(\phi_R-b_Rw)^2+(\phi_I-b_Iw)^2].
\end{aligned}\]
Therefore we deduce from
\eqref{RI++} that
\begin{equation}\label{RI+++}
\begin{aligned}
L_1(\phi^{\perp}_R, \phi^{\perp}_R)+L_1(\phi^{\perp}_I, \phi^{\perp}_I)+\lambda_R\int_{\R^2}[(\phi_R-b_Rw)^2+(\phi_I-b_Iw)^2] =0,
\end{aligned}
\end{equation}
and so
\begin{equation}
\lambda_R\int_{\R^2}[(\phi_R-b_Rw)^2+(\phi_I-b_Iw)^2]+a_1(\|\phi_R^{\perp}\|^2_{L^2(\R^2)}+\|\phi_I^{\perp}\|^2_{L^2(\R^2)})\leq 0.
\end{equation}

If $\lambda_R>0$, then we have $\phi_R=b_Rw$ and $\phi_I=b_Iw$. Putting $\phi_R$ and $\phi_I$ into equations \eqref{RP} and \eqref{IP} we get the linear system of $(b_R, b_I)$:
\[\left\{\begin{aligned}&\lambda_Rb_R-\lambda_Ib_I=0,\\
&\lambda_Rb_I-\lambda_Ib_R=0.\end{aligned}\right.
\]
Clearly $b_R=b_I=0$ and so $\phi=0$. We have a contradiction.

If $\lambda_R=0$, then  we have $\phi_R^{\perp}=\phi_I^{\perp}=0$. Putting $\phi_R$ and $\phi_I$ into equations \eqref{RP} and \eqref{IP}, using the facts $w_{y_j}\in Kernal(L_0)$ and $L_0[w_0]=w$, we get
\[\left\{\begin{aligned}&c_Rw=-\lambda_I\left(b_Iw+c_Iw_0+\sum_{j=1}^2d_{Ij}w_{y_j}\right),\\
&c_Iw=\lambda_I\left(b_Rw+c_Rw_0+\sum_{j=1}^2d_{Rj}w_{y_j}\right).\end{aligned}\right.
\]
Suppose $\lambda_I\neq 0$, we must have $b_R=b_I=c_R=c_I=d_{R1}=d_{R2}=d_{I1}=d_{I2}=0$, contradicting to the assumption that $\phi$ is nontrivial. Therefore
$\lambda_\infty=\lambda_R+i\lambda_I=0$, a contradiction.

 The claim is proved.
\end{proof}

Next we discuss possible limits of $\tilde{\tau} \lambda$.

\noindent
{\bf Claim 3:} $ |\tilde{\tau} \lambda | \to +\infty$ as $\tilde{\tau} \to +\infty$.

\begin{proof}
Suppose the claim is not true. Then along a subsequence $\tilde{\tau} \lambda \to \mu_{\infty} \in {\mathbb C}$ as $\tilde{\tau}\to \infty$. By Claim 1 we arrive at the following equation
\begin{equation}
\label{3}
\Delta \phi -\phi+ 3w^2 \phi -\frac{3}{1+\mu_{\infty} } \frac{\int_{\R^2} w^2 \phi}{\int_{\R^2} w^3} w^3-\frac{2\mu_{\infty}}{1+\mu_{\infty}} \frac{\int_{\R^2} w \phi}{ \int_{\R^2} w^2} w^3 =0.
\end{equation}
Hence
$$\phi=\frac{3}{1+\mu_{\infty} } \frac{\int_{\R^2} w^2 \phi}{\int_{\R^2} w^3} L_0^{-1}[w^3] +\frac{2\mu_{\infty}}{1+\mu_{\infty}} \frac{\int_{\R^2} w \phi}{ \int_{\R^2} w^2} L_0^{-1}[w^3]. $$
Using $L^{-1}_0[w^3]=\frac{1}{2}w$ we obtain
\begin{equation}\label{phi1+}
 2 \phi= \frac{3}{1+\mu_{\infty} } \frac{\int_{\R^2} w^2 \phi}{\int_{\R^2} w^3}w +\frac{2\mu_{\infty}}{1+\mu_{\infty}} \frac{\int_{\R^2} w \phi}{ \int_{\R^2} w^2} w.
 \end{equation}
Set $$A=\int_{\R^2} w \phi, \qquad B= \int_{\R^2} w^2 \phi.$$ Multiplying \eqref{phi1+} by $w$ and integrating over $\R^2$ we obtain
\begin{equation}\label{AB1}
\left(\frac{2\mu_{\infty}}{1+\mu_{\infty}}-2\right)A+\frac{3}{1+\mu_{\infty}} \frac{\int_{\R^2} w^2}{\int_{\R^2} w^3} B=0.
\end{equation}\label{AB2}
 Multiplying \eqref{phi1+} by $w^2$ and integrating over $\R^2$ we obtain
\begin{equation}
\frac{2\mu_{\infty}}{1+\mu_{\infty}} \frac{\int_{\R^2} w^3}{\int_{\R^2} w^2} A+\left(\frac{3}{1+\mu_{\infty}}-2\right) B =0.
\end{equation}
For the linear system \eqref{AB1}, \eqref{AB2} to have a solution $(A, B)\neq (0, 0)$, we must have
\[
\left(\frac{2\mu_{\infty}}{1+\mu_{\infty}}-2\right)\cdot\left(\frac{3}{1+\mu_{\infty}}-2\right)-\frac{3}{1+\mu_{\infty}} \frac{\int_{\R^2} w^2}{\int_{\R^2} w^3}\cdot\frac{2\mu_{\infty}}{1+\mu_{\infty}} \frac{\int_{\R^2} w^3}{\int_{\R^2} w^2} =0,
\]
and so $\mu_{\infty}=-1$, which is impossible since $Re(\mu_{\infty})\geq 0$.

Therefore $A=B=0$. By \eqref{phi1+} we have $\phi=0$, a contradiction.

\end{proof}

Therefore $|\tilde{\tau} \lambda| \to +\infty, \lambda \to 0$. We see that $\phi \to \phi_0$ in $H^1(\R^2)$ which satisfies
\begin{equation}
\label{4}
\Delta \phi_0 -\phi_0+ 3w^2 \phi_0 -2 \frac{\int_{\R^2} w \phi_0}{ \int_{\R^2} w^2} w^3 =0
\end{equation}
and hence
\[
\phi_0= 2 \frac{\int_{\R^2} w \phi_0}{ \int_{\R^2} w^2} L_0^{-1}[w^3]=\frac{\int_{\R^2} w \phi_0}{ \int_{\R^2} w^2} w=Cw.
\]
Without loss of generality, we may assume that $C=1$.

Let us decompose
$$ \phi = w + \phi^\perp$$
with
\begin{equation}
\label{41}
 \int_{\R^2} w \phi^\perp= 0.
 \end{equation}
In this way,  $\phi^\perp \to 0 $ in $H^1(\R^2)$ as $\tilde{\tau} \to +\infty$.

We then have
\begin{equation}
\label{5}
\Delta \phi^\perp -\phi^\perp+ 3w^2 \phi^\perp -\frac{3}{1+\tilde{\tau} \lambda} \frac{\int_{\R^2} w^2 \phi^\perp}{\int_{\R^2} w^3} w^3=\lambda \phi^\perp + \frac{1}{1+\tilde{\tau} \lambda}  w^3 + \lambda w.
\end{equation}

\noindent
{\bf Claim 4:} $ \frac{1}{1+\tilde{\tau} \lambda}= o(\lambda)$.

\begin{proof}

Multiplying (\ref{5}) by $w_0=L_0^{-1} [w]=\frac{1}{2}w+\frac{1}{2} y\nabla w $ and using (\ref{41}) we get
\begin{equation}
\label{7}
 \frac{1}{1+\tilde{\tau} \lambda} \int_{\R^2}  w^3 w_0 = - \lambda \int_{\R^2}\phi^\perp w_0 -\frac{3}{1+\tilde{\tau} \lambda} \frac{\int_{\R^2} w^2 \phi^\perp}{\int_{\R^2} w^3} \int_{\R^2} w^3 w_0,
 \end{equation}
where we have used the identities \[\int_{\R^2} ww_0=0,\qquad\int_{\R^2} w_0L_0[\phi^{\perp}]=\int_{\R^2}\phi^{\perp} L_0[w_0]=\int_{\R^2}\phi^{\perp} w=0.\]
Note \[ \| \phi^\perp\|_{L^2(\R^2)}= o(1),\]
\[\int_{\R^2} w^3w_0=\frac{1}{2}\int_{\R^2} w^3(w+y\nabla w)=\frac{1}{4}\int_{\R^2} w^4>0.\]
Putting the last two identities into \eqref{7} we obtain
\[\frac{1+o(1)}{1+\tilde{\tau}\lambda}\int_{\R^2}w^4=o(\lambda).\]
The claim is proved.
\end{proof}

\medskip
\noindent
{\bf Claim 5:} $ \phi^\perp = \lambda (1+ o(1)) w_0 $.

\begin{proof}
Since $ \frac{1}{1+\tilde{\tau} \lambda}= o(\lambda)$, we set $ \phi^\perp=\lambda \phi_1$, then
$$ \begin{aligned}L_0 \phi_1&=\frac{3}{1+\tilde{\tau} \lambda} \frac{\int_{\R^2} w^2 \phi_1}{\int_{\R^2} w^3} w^3+\lambda \phi_1+ \frac{1}{\lambda(1+\tilde{\tau}\lambda)}  w^3 +  w\\
 &=  o(1)w^3 + o(1) \phi_1+w \end{aligned}$$
Since $ L_0^{-1}$ exists we have
\[
\phi_1=L_0^{-1}[w]+o(1)(L_0^{-1}[w^3]+L_0^{-1}[\phi_1])=(1+o(1))w_0.
\]
The claim is proved.
\end{proof}

Finally we derive the equation for $\lambda$: from (\ref{7}) we get
\[
\frac{1+o(1)}{4(1+\tilde{\tau}\lambda)}\int_{\R^2}w^4=-\lambda^2\int_{\R^2}(w_0)^2-\frac{3\lambda}{1+\tilde{\tau}\lambda}\frac{\int_{\R^2}w^2w_0}{\int_{\R^2}w^3}\int_{\R^2}w^3w_0.
\]
We calculate
\[
\int_{\R^2}w^2w_0=\frac{1}{2}\left(\int_{\R^2}w^3+\int_{\R^2}w^2y\nabla w\right)=\frac{1}{6}\int_{\R^2}w^3.
\]
Together with the two already known identities
\[
\int_{\R^2}w^3w_0=\frac{1}{4}\int_{\R^2}w^4,\quad \int_{\R^2}w^4=2\int_{\R^2}w^2
\]
we obtain
$$ -\frac{1}{2}[1+o(1)] \int_{\R^2} w^2 =(\tilde{\tau}\lambda^3+\lambda^2) \int_{\R^2} (w_0)^2.   $$
Therefore, since $\lambda^2=o(1)$, $\lambda$ must satisfy the following algebraic equation
$$ \lambda^{3}= -[1+o(1)]\frac{c_0}{\tilde{\tau}}$$
where $c_0= \frac{\int_{\R^2} w^2}{2\int_{\R^2} (w_0)^2}>0$.

Since $ Re (\lambda) \geq 0$, we obtain two conjugate solutions \[ \lambda \sim \tilde{\tau}^{-1/3}c_0^{1/3} e^{\frac{\pi}{3} i}\qquad\text{or}\qquad\lambda= \tilde{\tau}^{-1/3} c_0^{1/3} e^{-\frac{\pi}{3} i}.\] We see that \[ Re(\lambda)=\lambda_R \sim \frac{1}{2} \tilde{\tau}^{-1/3} c_0^{1/3} >0.\]

Conversely we can also easily construct eigenvalues with $ \lambda \sim \tilde{\tau}^{-1/3}c_0^{1/3} e^{\frac{\pm\pi}{3} i}$.

\medskip
\noindent
{\bf Claim 6:} There exists a Hopf bifurcation at some $\tilde{\tau}_h>0$.

\begin{proof}
This claim can be proved by using a continuation argument of Dancer \cite{Dancer2001}. As in Dancer \cite{Dancer2001}, we may only consider radial eigenfunctions. Then $0$ is not an eigenvalue of \eqref{NLEP++}. If $\tilde{\tau}=0$, by the remark of Theorem 1 of \cite{Wei2001}, all the eigenvalues of \eqref{NLEP++} has negative real parts. By Claim 5, there exists some $\tilde{\tau}_*>0$ large enough such that \eqref{NLEP++} has an eigenvalue with positive real part. Therefore there is some $\tilde{\tau}_h\in (0, \tilde{\tau}_*)$, \eqref{NLEP++} has a pair of conjugate pure imaginary eigenvalues.
\end{proof}

\begin{rem}
The argument in this section does not restrict to radial eigenfunctions.
\end{rem}

\section{The stability of the spike solutions in the one-dimensional case}\label{sec: NLEP2}

When the space dimension  $N=1$. In the near shadow case, as in the two-dimensional case, the stability of the original system is determined by the following nonlocal eigenvalue problem
\begin{equation}
\label{1D1}
L_0\phi-\frac{3}{1+\tilde{\tau} \lambda} \frac{\int_{\R} w^2 \phi}{\int_{\R} w^3} w^3-\frac{2\tilde{\tau} \lambda}{1+\tilde{\tau} \lambda} \frac{\int_{\R} w \phi}{ \int_{\R} w^2} w^3 =\lambda \phi,\qquad \phi\in H^1(\R),
\end{equation}
where
\begin{equation}
L_0\phi:=\Delta \phi -\phi+ 3w^2 \phi,
\end{equation}
and
\begin{equation}
\tilde{\tau}=\frac{\ve\tau(\ve)}{\gamma_0|B_R|}\int_{\R}w^2\,dy.
\end{equation}

If $\tilde{\tau}=0$, by the remark of Theorem 1 of \cite{Wei2001}, any eigenvalue of \eqref{1D1} must satisfy $ Re (\lambda)<0$.  Hence for $\tilde{\tau} $ small problem \eqref{1D1} is stable.

On the other hand we have the following result for the large $\tilde{\tau}$ case.

\begin{thm}\label{thm6.1}
There exists $\tilde{\tau}_0>0$ such that for $\tilde{\tau} >\tilde{\tau}_0$, any eigenvalue of \eqref{1D1} must satisfy $ Re (\lambda) <0$.
\end{thm}

We prove this theorem through a series of claims.
From now on we may assume that $\tilde{\tau}>0$ is large and $ \lambda=\lambda_R+i \lambda_I$ with $\lambda_R \geq 0$.

We first claim:

\medskip

\noindent
{\bf Claim 1:} If $\lambda_R \geq 0$, then $ |\lambda|\leq C$ for some positive constant  independent of $\tilde{\tau}$.

\begin{proof}

Multiplying \eqref{1D1} by $\bar{\phi}$ (the complex conjugate of $\phi$), and integrating the resultant equation over $\R$ we obtain
\begin{equation}\label{1D1+}
\begin{aligned}
 &\int_{\R} (|\nabla \phi|^2 +|\phi|^2 - 3 w^2 |\phi|^2) + \lambda \int_{\R} |\phi|^2 \\
 &= -\frac{3}{1+\tilde{\tau} \lambda} \frac{\int_{\R} w^2 \phi}{\int_{\R} w^3 } \int_{\R} w^3\bar{\phi}-\frac{2\tilde{\tau} \lambda}{1+\tilde{\tau} \lambda} \frac{\int_{\R} w \phi}{ \int_{\R} w^2} \int_{\R} w^3\bar{\phi}.
\end{aligned}
\end{equation}
Let $\mu_0>0$ be the first eigenvalue of $L_0$ given in Lemma \ref{lem:L0}. We have by the  variational representation of $\mu_0$ that   $$\int_{\R} (|\nabla \phi|^2 +|\phi|^2 - 3 w^2 |\phi|^2)\geq -\mu_0 \int_{\R} |\phi|^2.$$
The integrals on the right-hand side of \eqref{1D1+} can be estimated using the Holder inequalities:
\[
\left|\int_{\R} w^{k}\phi\right|\leq \left(\int_{\R}w^{2k}\right)^{1/2}\left(\int_{\R}|\phi|^2\right)^{1/2},\;\;k=1, 2,\]
and
\[\left|\int_{\R} w^3\bar{\phi}\right|\leq \left(\int_{\R}w^2\right)^{1/2}\left(\int_{\R}|\phi|^2\right)^{1/2}.
\]
It follows that
$$ |\lambda| \leq \mu_0  + C \left(\Big|\frac{3}{1+\tilde{\tau} \lambda}\Big|+ \Big|\frac{2\tilde{\tau} \lambda}{1+\tilde{\tau} \lambda}\Big|\right),$$
which together with $\lambda_R\geq 0$ implies $|\lambda|\leq C$.

The claim is proved.

\end{proof}

\medskip

\noindent
{\bf Claim 2:} If $\tilde{\tau} \to +\infty$  then $ \lambda \to 0$.

\begin{proof}
Suppose the claim is false. We have, along a subsequence,  \[\tilde{\tau}\to \infty,\quad \lambda \to \lambda_\infty \not = 0.\] Then \[\frac{1}{1+\tilde{\tau} \lambda} \to 0,\] and we obtain the following limiting problem
\begin{equation}
\label{1D2}
\Delta \phi -\phi+ 3w^2 \phi -2 \frac{\int_{\R} w \phi}{ \int_{\R} w^2} w^3 =\lambda_\infty \phi.
\end{equation}

A contradiction can then be derived by following the same line of the proof of Case 1 of Theorem 1.4 in \cite{Wei1999}, with a few modifications.

Let the linear operator $L_1:H^1(\R)\to L^2(\R)$ be defined by
\begin{equation*}
L_1\phi:=L_0\phi-2\frac{\int_{\R} w\phi}{\int_{\R} w^2}w^3-2\frac{\int_{\R} w^3\phi}{\int_{\R} w^2}w+2\frac{\int_{\R} w^{4}\int_{\R} w\phi}{\left(\int_{\R} w^2\right)^2}w,\qquad \phi\in H^1(\R).
\end{equation*}
Then $L_1$ is self-adjoint.

According to Lemma 5.1 of \cite{Wei1999}, we have
\begin{itemize}
\item The kernel of $L_1$ is given by $X_1=\textnormal{span}\{w, w_y\}$.
    \item There exists a positive constant $a_1>0$ such that for all $\phi\in H^1(\R^2)$
   \begin{equation}
   \label{1DL1}
   \begin{aligned}
    L_1(\phi, \phi)=L_0(\phi, \phi)&
    +4\frac{\int_{\R} w\phi\int_{\R} w^3\phi}{\int_{\R} w^2}-2\frac{\int_{\R} w^{4}}{\left(\int_{\R} w^2\right)^2}\left(\int_{\R} w\phi\right)^2\\
    &\geq a_1d^2_{L^2(\R)}(\phi, X_1),
    \end{aligned}
    \end{equation}
    where
    \[
    L_0(\phi, \phi):=\int_{\R}(|\nabla \phi|^2+\phi^2-3w^2\phi^2),
    \]
    and
    $d_{L^2(\R)}(\phi, X_1)$ is the distance of $\phi$ to $X_1$ in the space of $L^2(\R)$.
\end{itemize}

Now we are ready to prove the claim. Let $\lambda_\infty=\lambda_R+i\lambda_I$ and $\phi=\phi_R+i\phi_I$. Then we have the system of equations
\begin{equation}\label{1DRP}
L_0\phi_R-2 \frac{\int_{\R} w \phi_R}{ \int_{\R} w^2} w^3 =\lambda_R \phi_R-\lambda_I\phi_I,
\end{equation}
\begin{equation}\label{1DIP}
L_0\phi_I-2 \frac{\int_{\R} w \phi_I}{ \int_{\R} w^2} w^3 =\lambda_R \phi_I+\lambda_I\phi_R.
\end{equation}
Multiplying \eqref{1DRP} by $\phi_R$, \eqref{1DIP} by $\phi_I$, integrating over $\R$, and summing up, we obtain
\begin{equation}
\begin{aligned}
L_0(\phi_R, \phi_R)&+L_0(\phi_I, \phi_I)+\lambda_R\int_{\R} (\phi_R^2+\phi_I^2)\\
&=-\frac{2}{\int_{\R} w^2}\left(\int_{\R} w^3\phi_R\int_{\R} w\phi_R+\int_{\R} w^3\phi_I\int_{\R} w\phi_I\right),
\end{aligned}
\end{equation}
or in the form
\begin{equation}\label{1DRI}
\begin{aligned}
&L_1(\phi_R, \phi_R)+L_1(\phi_I, \phi_I)+\lambda_R\int_{\R} (\phi_R^2+\phi_I^2)\\
&\qquad=2\frac{\int_{\R} w\phi_R\int_{\R} w^3\phi_R+\int_{\R} w\phi_I\int_{\R} w^3\phi_I}{\int_{\R} w^2}\\
&\qquad-\frac{2\int_{\R} w^{4}}{\left(\int_{\R} w^2\right)^2}\left[\left(\int_{\R} w\phi_R\right)^2+\left(\int_{\R} w\phi_I\right)^2\right].
\end{aligned}
\end{equation}
Multiplying \eqref{RP} and \eqref{IP} by $w$ respectively, and integrating over $\R$, we obtain
\begin{equation}\label{1DRP+}
2\int_{\R} w^3\phi_R-2 \frac{\int_{\R} w^4}{ \int_{\R} w^2}\int_{\R} w \phi_R =\lambda_R\int_{\R}  w\phi_R-\lambda_I\int_{\R} w\phi_I,
\end{equation}
\begin{equation}\label{1DIP+}
2\int_{\R} w^3\phi_I-2 \frac{\int_{\R} w^4}{ \int_{\R} w^2}\int_{\R} w \phi_I =\lambda_R\int_{\R}  w\phi_I+\lambda_I\int_{\R} w\phi_R.
\end{equation}
Multiplying \eqref{1DRP+} by $\int_{\R} w\phi_R$, \eqref{1DIP+} by $\int_{\R} w\phi_I$, and summing up, we obtain
\begin{equation}\label{1DRI+}
\begin{aligned}
&\int_{\R} w\phi_R\int_{\R} w^3\phi_R+\int_{\R} w\phi_I\int_{\R} w^3\phi_I\\
&=\left(\frac{\lambda_R}{2}+\frac{\int_{\R} w^4}{ \int_{\R} w^2}\right)\left[\left(\int_{\R} w\phi_R\right)^2+\left(\int_{\R} w\phi_I\right)^2\right].
\end{aligned}
\end{equation}
Plugging \eqref{1DRI+} into \eqref{1DRI} we obtain
\begin{equation}\label{1DRI++}
\begin{aligned}
&L_1(\phi_R, \phi_R)+L_1(\phi_I, \phi_I)+\lambda_R\int_{\R} (\phi_R^2+\phi_I^2)\\
&=\frac{\lambda_R}{\int_{\R} w^2}\left[\left(\int_{\R} w\phi_R\right)^2+\left(\int_{\R} w\phi_I\right)^2\right].
\end{aligned}
\end{equation}
Decompose
\[
\phi_R=c_Rw+d_Rw_y+\phi_R^{\perp},\quad\phi_R^{\perp}\perp X_1,
\]
\[
\phi_I=c_Iw+d_Iw_y+\phi_I^{\perp},\quad\phi_I^{\perp}\perp X_1,
\]
and then  put them into \eqref{1DRI++} and calculate
\[
\left(\int_{\R} w\phi_R\right)^2+\left(\int_{\R} w\phi_I\right)^2=(c_R^2+c_I^2)\left(\int_{\R}w^2\right)^2,
\]
\[
\int_{\R}(\phi_R^2+\phi_I^2)=(c_R^2+c_I^2)\int_{\R}w^2+(d_R^2+d_I^2)\int_{\R}w_y^2+\int_{\R}[(\phi_R^{\perp})^2+(\phi_I^{\perp})^2].
\]
We therefore deduce from \eqref{1DRI++}
\begin{equation}\label{1DRI+++}
\begin{aligned}
&L_1(\phi^{\perp}_R, \phi^{\perp}_R)+L_1(\phi^{\perp}_I, \phi^{\perp}_I)\\
&+\lambda_R(c_R^2+c_I^2)\int_{\R}w^2+\lambda_R(\|\phi_R^{\perp}\|^2_{L^2(\R)}+\|\phi_I^{\perp}\|^2_{L^2(\R)}) =0,
\end{aligned}
\end{equation}
and so by \eqref{1DL1}
\begin{equation*}
\begin{aligned}
&\lambda_R(c_R^2+c_I^2)\int_{\R} w^2+\lambda_R(d_R^2+d_I^2)\int_{\R} w_y^2\\
&+(\lambda_R+a_1)(\|\phi_R^{\perp}\|^2_{L^2(\R)}+\|\phi_I^{\perp}\|^2_{L^2(\R)})\leq 0.
\end{aligned}
\end{equation*}

If $\lambda_R>0$, then we have $\phi_R^{\perp}=\phi_I^{\perp}=0$ and $c_R=c_I=d_R=D_I=0$  so $\phi_R=\phi_I=0$. We have a contradiction.

If $\lambda_R=0$, then  we have $\phi_R^{\perp}=\phi_I^{\perp}=0$. Putting $\phi_R$ and $\phi_I$ into equations \eqref{1DRP} and \eqref{1DIP} and using the identities
\[
L_0w_y=0\qquad\text{and}\;\;\;L_0w=2w^3
\]
we get
\[
\begin{aligned}
&\lambda_I(c_I w+d_I w_y)=0,\\
&\lambda_I(c_R w+d_R w_y)=0.
\end{aligned}
\]
If $\lambda_I\neq 0$ we then have $c_R=c_I=d_R=d_I=0$ and hence $\phi_R=\phi_I=0$, a contraction.

Therefore $\lambda_\infty=0$. The claim is proved.
\end{proof}

Next we discuss possible limits of $\tilde{\tau} \lambda$.

\noindent
{\bf Claim 3:} $ |\tilde{\tau} \lambda | \to +\infty$ as $\tilde{\tau} \to +\infty$.

\begin{proof}
Suppose for contradiction the claim is false. Then we may assume that along a subsequence  $\tilde{\tau} \lambda \to \mu_\infty \in {\mathbb C}$. Then by Claim 1 we arrive at the following equation
\begin{equation}
\label{1D3}
\Delta \phi -\phi+ 3w^2 \phi -\frac{3}{1+\mu_\infty } \frac{\int_{\R} w^2 \phi}{\int_{\R} w^3} w^3-\frac{2\mu_\infty}{1+\mu_\infty} \frac{\int_{\R} w \phi}{ \int_{\R} w^2} w^3 =0.
\end{equation}

Hence
$$ \phi=\frac{3}{1+\mu_\infty } \frac{\int_{\R} w^2 \phi}{\int_{\R} w^3} L_0^{-1}[ w^3]+\frac{2\mu_\infty}{1+\mu_\infty} \frac{\int_{\R} w \phi}{ \int_{\R} w^2} L_0^{-1}[w^3], $$
and so
\begin{equation}\label{phi2}
2 \phi= \left[\frac{3}{1+\mu_\infty } \frac{\int_{\R} w^2 \phi}{\int_{\R} w^3} +\frac{2\mu_\infty}{1+\mu_\infty} \frac{\int_{\R} w \phi}{ \int_{\R} w^2} \right] w.
\end{equation}
Let $A=\int_{\R} w \phi, B= \int_{\R} w^2 \phi$. Multiplying \eqref{phi2} by $w$ and integrating over $\R$ we obtain
$$ 2A = \frac{3}{1+\mu_\infty} \frac{\int_{\R} w^2}{\int_{\R} w^3} B + \frac{2 \mu_\infty}{1+\mu_\infty} A, $$
Multiplying \eqref{phi2} by $w^2$ and integrating over $\R$ we obtain
$$ 2B= \frac{3}{1+\mu_\infty} B + \frac{2\mu_\infty}{1+\mu_\infty} \frac{\int_{\R} w^3}{\int_{\R} w^2} A.$$
It follows that
$$ 2B=3B,$$
which is impossible.
Hence $A=B=0$ and $\phi=0$. The contradiction finishes the proof of the claim.
\end{proof}

Therefore $\tilde{\tau} \lambda \to +\infty, \lambda \to 0$. We see that $\phi \to \phi_0$  in $H^1(\R)$ which satisfies
\begin{equation}
\label{1D4}
\Delta \phi_0 -\phi_0+ 3w^2 \phi_0 -2 \frac{\int_{\R} w \phi_0}{ \int_{\R} w^2} w^3 =0
\end{equation}
and hence $\phi_0= C w$. (We may assume that $C=1$).

Let us decompose
$$ \phi = w + \phi^\perp$$
with
\begin{equation}
\label{1D41}
 \int_{\R} w \phi^\perp= 0.
 \end{equation}
In this way,  $\phi^\perp \to 0 $ in $L^2(\R)$ as $\tilde{\tau} \to +\infty$.

We then obtain from \eqref{1D1}
\begin{equation}
\label{1D5}
\Delta \phi^\perp -\phi^\perp+ 3w^2 \phi^\perp -\frac{3}{1+\tilde{\tau} \lambda} \frac{\int_{\R} w^2 \phi^\perp}{\int_{\R} w^3} w^3=\lambda \phi^\perp + \lambda w + \frac{1}{1+\tilde{\tau} \lambda}  w^3.
\end{equation}

\noindent
{\bf Claim 4:} $ \frac{2}{1+\tilde{\tau} \lambda} +\lambda=o(\lambda)$.

\begin{proof}

Multiplying (\ref{5}) by $w_0=L_0^{-1} [w]=\frac{1}{2}w+\frac{1}{2} y w_y $ and using (\ref{1D41}) we get
\begin{equation}
\label{1D7}
 \frac{1}{1+\tilde{\tau} \lambda}\left(1+3\frac{\int_{\R} w^2 \phi^\perp}{\int_{\R} w^3} \right) \int_{\R}  w^3 w_0+ \lambda \left(\int_{\R} w w_0 +\int_{\R} \phi^\perp w_0\right)=0.
 \end{equation}
On the one hand since $ \| \phi^\perp\|_{L^2(\R)}= o(1)$, we have \[3\frac{\int_{\R} w^2 \phi^\perp}{\int_{\R} w^3}=o(1),\qquad\int_{\R} \phi^\perp w_0=o(1).\]
On the other hand, since
\[
\int_{\R}w^4=\frac{4}{3}\int_{\R}w^2,\qquad \int_{\R}w^2_y=\frac{1}{3}\int_{\R}w^2.
\]
we have
\begin{equation}\label{1Dw3w0}
 \begin{aligned}
 \int_{\R} w^3 w_0&= \frac{1}{2} \left(\int_{\R} w^4+\int_{\R}yw_yw^3\right)\\
&=\frac{1}{2} \left(\int_{\R} w^4-\frac{1}{4}\int_{\R}w^4\right)\\
&=\frac{1}{2}\int_{\R}w^2,
\end{aligned}
\end{equation}
and
\begin{equation}\label{1Dww0}
\int_{\R} w w_0= \frac{1}{2} \left(\int_{\R} w^2+\int_{\R}yw_yw\right)=\frac{1}{4} \int_{\R} w^2.
\end{equation}
Therefore from (\ref{1D7}) we have
\[ \frac{2}{1+\tilde{\tau} \lambda}  + \lambda =o(\lambda). \]
As a result,
\begin{equation}\label{V-lambda1}
\lambda=\pm\sqrt{2} \tilde{\tau}^{-1/2} i +  O(\tilde{\tau}^{-1}).
\end{equation}
\end{proof}

The estimate \eqref{V-lambda1} is not sufficient to determine the sign of $Re(\lambda)$. We proceed to find the next order of $Re(\lambda)$.

\medskip
\noindent
{\bf Claim 5:} $ \phi^\perp = \lambda (1+ o(1)) \phi_0 $ where $\phi_0$ satisfies
$$ L_0 \phi_0= w -\frac{1}{2} w^3. $$
and hence
$$ \phi_0= w_0-\frac{1}{4} w.$$

\begin{proof}
Since $ \frac{2}{1+\tilde{\tau} \lambda} \sim -\lambda$, we set $ \phi^\perp=\lambda \phi_1$, plug into \eqref{1D5}, and get
$$ L_0 [\phi_1] =\frac{3}{1+\tilde{\tau}\lambda}\frac{\int_{\R}w^2\phi_1}{\int_{\R}w^3}w^3+\lambda\phi_1+ w -\frac{1}{2}w^3+ o(1). $$
Since
\[
\frac{3}{1+\tilde{\tau}\lambda}\frac{\int_{\R}w^2\phi_1}{\int_{\R}w^3}w^3=o(1), \quad \lambda\phi_1=o(1),
\]
and $ L_0^{-1}$ exists we obtain that
\[
\phi_1=(1+o(1))\phi_0,
\]
where $\phi_0$ is defined by
\[
L_0\phi_0=w-\frac{1}{2}w^3.
\]
So
\[
\phi_0=L^{-1}_0[w]-\frac{1}{2}L_0^{-1}[w^3]=w_0-\frac{1}{4}w.
\]
The claim is proved.

We note that $$ \int_{\R} w \phi_0=\int_{\R}ww_0-\frac{1}{4}\int_{\R}w^2=0.$$
\end{proof}

Finally we derive the equation for $\lambda$: from (\ref{1D7}) we get
\begin{equation}
\label{1D8}
\begin{aligned}
  &\frac{1}{1+\tilde{\tau} \lambda} \int_{\R}  w^3 w_0+ \lambda \int_{\R} w w_0\\
  & = -\lambda^2 \int_{\R} \phi_0 w_0 -\frac{3\lambda }{1+\tilde{\tau} \lambda} \frac{\int_{\R} w^2 \phi_0}{\int_{\R} w^3} \int_{\R} w^3 w_0 + o(\lambda^2).
  \end{aligned}
  \end{equation}
Note that from \eqref{1Dw3w0} and \eqref{1Dww0} we have
$$ \int_{\R}  w^3 w_0=\frac{1}{2} \int_{\R} w^2,\qquad \int_{\R} w w_0 =\frac{1}{4} \int_{\R} w^2.$$
On the other hand we have
$$ \int_{\R} \phi_0 w_0=\int_{\R} (w_0)^2 - \frac{1}{16} \int_{\R} w^2,$$
and
$$\begin{aligned}&\int_{\R} w^2 \phi_0= \int_{\R} w^2 (w_0-\frac{1}{4} w)\\
&=\frac{1}{4}\int_{\R}w^3+\frac{1}{2}\int_{\R}yw_yw^2\\
&= \frac{1}{12} \int_{\R} w^3.
\end{aligned} $$
Substituting the above into (\ref{1D8}), we get
$$ \lambda= \sqrt{2} \tilde{\tau}^{-1/2} i +  \tilde{\tau}^{-1} b, $$
where
$$ \lambda^2 \sim -2 \tilde{\tau}^{-1}, \qquad\frac{3\lambda }{1+\tilde{\tau} \lambda} \sim 3 \tilde{\tau}^{-1}.$$
Therefore we have the equation for $b$:
$$ \frac{ b+1}{2}\times\frac{1}{2} \int_{\R} w^2 + b\times \frac{1}{4} \int_{\R} w^2 = 2 \int_{\R} \phi_0 w_0-3\times\frac{1}{12} \times\frac{1}{2} \int_{\R} w^2 +o(1),$$
From here we obtain
$$ b=4\frac{\int_{\R}\phi_0 w_0}{\int_{\R}w^2}-\frac{3}{4}+o(1). $$
Using the expression $w_0=\frac{1}{2}(w+yw_y)$ we have
\begin{equation}
\begin{aligned}
b&=\frac{\int_{\R} \phi_0w_0}{\int_{\R}w^2}-\frac{3}{4}+o(1)=\int_{\R} (w_0)^2-1+o(1)\\
&=\frac{1}{\int_{\R}w^2}\left(\int_{\R} y^2(w_y)^2+2\int_{\R}yw_yw+\int_{\R}w^2\right)-1+o(1)\\
&=\frac{\int_{\R}y^2(w_y)^2}{\int_{\R}w^2}-1+o(1),
\end{aligned}
\end{equation}
To further simplify the form of $b$ we note that in the one dimensional case
\[
w(y)=\sqrt{2}\sech y.
\]
It is clear that
\[
\int_{\R}w^2=4.
\]
We can also compute (see the Appendix) that
\[\int_{\R}y^2(w_y)^2=\frac{8}{3}+\frac{\pi^2}{9}.\]
Therefore
\[
b=\frac{\pi^2-12}{36}+o(1)<0.
\]

The proof of Theorem \ref{thm6.1} is now complete.

\begin{rem}
The argument in this section does not restrict to even eigenfunctions.
\end{rem}

\bigskip

\section{Appendix: The computation of $\int_{\R} (y^2 w_y^2) dy$}

Since \[
w(y)=\sqrt{2}\sech y,
\]
we have
\[w_y=2\sqrt{2}\times\frac{e^{-y}-e^y}{\left(e^{y}+e^{-y}\right)^2},\]
and so
\[
\int_{-\infty}^{\infty} y^2(w_y)^2=16\int_0^\infty y^2\frac{(e^{-y}-e^y)^2}{\left(e^{y}+e^{-y}\right)^4}dy.
\]
We compute
\[\begin{aligned}
&\int \frac{(e^{-y}-e^y)^2}{\left(e^{y}+e^{-y}\right)^4}dy=-\frac{1}{3}\int (e^y-e^{-y})d(e^y+e^{-y})^{-3}\\
&=-\frac{1}{3}(e^y-e^{-y})(e^y+e^{-y})^{-3}+\frac{1}{3}\int(e^y+e^{-y})^{-2}dy\\
&=-\frac{1}{3} (e^y-e^{-y})(e^y+e^{-y})^{-3}-\frac{1}{6}e^{-y}(e^y+e^{-y})^{-1}+C.
\end{aligned}\]
Therefore
\[\begin{aligned}
&\int_0^\infty y^2\frac{(e^{-y}-e^y)^2}{\left(e^{y}+e^{-y}\right)^4}dy=\frac{2}{3}\int_0^{\infty} y (e^y-e^{-y})(e^y+e^{-y})^{-3}dy+\frac{1}{3}\int_0^{\infty}ye^{-y}(e^y+e^{-y})^{-1}dy.
\end{aligned}\]
We have
\[
\int (e^y-e^{-y})(e^y+e^{-y})^{-3}dy=-\frac{1}{2}(e^y+e^{-y})^{-2}+C,
\]
and
\[
\int e^{-y}(e^y+e^{-y})^{-1}dy=-\frac{1}{2}\log\left(1+e^{-2y}\right)+C.
\]
So
\[
\int_0^{\infty} y (e^y-e^{-y})(e^y+e^{-y})^{-3}dy=\frac{1}{2}\int_0^\infty (e^y+e^{-y})^{-2}dy=\frac{1}{4}.
\]
While
\[\begin{aligned}
&\int_0^{\infty}ye^{-y}(e^y+e^{-y})^{-1}dy=\frac{1}{2}\int_0^{\infty}\log\left(1+e^{-2y}\right)dy\\
&=\frac{1}{4}\int_0^{\infty}\log\left(1+e^{-y}\right)dy=\frac{1}{4}\int_0^1\frac{\log(1+t)}{t}dt,
\end{aligned}\]
where
\[
\begin{aligned}
&\int_0^1\frac{\log(1+t)}{t}=\int_0^1\left(1-\frac{t}{2}+\frac{t^2}{3}-\cdots\right)dt\\
&=1-\frac{1}{2^2}+\frac{1}{3^2}-\frac{1}{4^2}+\cdots\\
&=(1+\frac{1}{2^2}+\frac{1}{3^2}+\frac{1}{4^2}+\cdots)-\frac{2}{2^2}(1+\frac{1}{2^2}+\frac{1}{3^2}+\frac{1}{4^2}+\cdots)\\
&=\frac{\pi^2}{12}.
\end{aligned}\]
Therefore finally we obtain
\begin{equation}
\int_{\R}y^2(w_y)^2=\frac{8}{3}+\frac{\pi^2}{9}.
\end{equation}

\end{document}